\documentclass[smallextended]{svjour3}       
\smartqed  

\usepackage{amssymb,amsmath,amsfonts,amsxtra}
\usepackage{bm,mathrsfs,braket}
\usepackage{mathtools,nccmath,bbm,upgreek}
\usepackage{dsfont}

\usepackage{graphicx,epsfig,epstopdf}
\usepackage{tikz}
\usetikzlibrary{arrows,calc,decorations.pathreplacing,arrows.meta,positioning}
\graphicspath{{./pics/}}

\usepackage{booktabs,array,multirow,threeparttable}

\usepackage{algorithm,algorithmic}
\usepackage{verbatim,comment}

\usepackage{caption}
\usepackage{subcaption}

\usepackage[colorlinks,linktocpage,linkcolor=blue]{hyperref}

\usepackage{url,times,marvosym,eurosym}
\usepackage{float,enumerate}

\allowdisplaybreaks

\renewcommand{\d}{{\rm d}}

\newtheorem{assumption}{Assumption}
\numberwithin{equation}{section}

\def\d{{\mathrm d}}

\def\0{{\textbf{0}}}

\def\gl{\gamma_\mathrm{lin}}

\def\<{{\langle }}
\def\>{{\rangle }}
\def\II{(\Om)}
\def\dH#1{\dot H^{#1}(\Omega)}

\def\Dom{\text{Dom}}

\newcommand{\fy}{\varphi}

\def\II{(\Omega)}

\begin{document}

\title{Linear Convergence of Parareal Algorithm for Semilinear Parabolic Equations\thanks{G. Li is supported by Hong Kong Research Grants Council (Project 17317022). S. Wu is supported by National Natural Science Foundation of China (12171080,
12292982) and by Natural Science Foundation of Jilin Province (JC010284408).
Z. Zhou is supported by National Natural Science Foundation of China (12422117, 12426312).}}

\titlerunning{Linear Convergence of Parareal Algorithm for Semilinear Parabolic Equations}        

\author{Guanglian Li \and Qingle Lin \and  Shu-lin Wu \and Zhi Zhou}

\authorrunning{G. Li, Q. Lin, S-L. Wu, and Z. Zhou} 

\institute{G. Li \at
              Department of Mathematics, The University of Hong Kong, Pok Fu
              Lam Road, Hong Kong SAR, China. \\
              \email{lotusli@maths.hku.hk}           
           \and
          Q. Lin, Z. Zhou \at
           Department of Applied Mathematics, The Hong Kong Polytechnic University, Kowloon, Hong Kong, P. R. China.\\
          \email{qingle.lin@connect.polyu.hk, zhizhou@polyu.edu.hk}
           \and
          S-L. Wu \at
           College of Mathematics and Statistics, Northeast Normal University, Changchun, China.\\
           \email{wusl393@nenu.edu.cn}
}

\date{Received: date / Accepted: date}

\maketitle

\begin{abstract}
Long-time simulations of evolution equations present substantial computational challenges due to the inherently sequential nature of conventional time-stepping schemes. The parareal method, a leading parallel-in-time (PinT) algorithm, offers a promising approach to overcome the challenge by introducing concurrency in the time domain. While its convergence theory is well-established for linear problems, extending the theory to nonlinear problems, particularly when the problem data have only limited regularity, remains a significant challenge. 
In this work, we provide the convergence analysis of the parareal algorithm for solving semilinear parabolic equations with an $H^2$ initial data. We employ stable rational approximations and first-order linearization as coarse propagators,  establish the linear convergence of the parareal algorithm and provide a sharp estimate for the convergence factor. The analysis combines the error-splitting technique from the superlinear convergence analysis of the parareal method, a refined linear convergence theory for linear parabolic equations, and \textsl{a priori} error estimates that are optimal with respect to the regularity of the problem data. The analysis shows the close connection between the convergence behavior of nonlinear models and their linear counterparts. Numerical experiments fully support the theoretical findings.

\vskip5pt

\end{abstract}
 
 \section{Introduction}
 
 In this work, we aim to establish a sharp convergence estimate for the parareal algorithm for semilinear parabolic problems.
 Let $T > 0$ be a fixed terminal time and $\Omega\subset \mathbb{R}^d$ ($d=1,2,3$) be a convex polygonal domain. Let $A$ be a self-adjoint second-order elliptic operator with $\Dom(A) = H^2(\Omega)\cap H_0^1(\Omega)$, $u_0 \in \Dom(A)$ the initial data and $f$  a given nonlinear reaction term. Consider the following semilinear parabolic problem for $u \in C((0,T]; \Dom(A)) \cap C([0,T]; L^2(\Omega))$:
 \begin{equation}\label{eqn:pde}
 	\left\{
 	\begin{aligned}
 		u'(t) + A u(t) &= f(u(t)), \quad 0 < t < T, \\
 		u(0) &= u_0.
 	\end{aligned}
 	\right.
 \end{equation}
 The numerical solution of the evolution model  \eqref{eqn:pde} often employs time-stepping schemes. These schemes typically advance sequentially, step by step, which can lead to significant computational bottlenecks, particularly for long-time simulations. The past few decades have witnessed growing interest in Parallel-in-Time (PinT) algorithms that enjoy parallelism in the time domain. Comprehensive discussions on PinT methods can be found in the monograph \cite{martin-pint} and the reviews \cite{GWZ:Acta,gander201550}.
 
 In the seminal work \cite{LionsMadayTurinici:2001}, Lions, Maday, and Turinici proposed the parareal algorithm for solving evolution models in a parallel-in-time manner. It works by dividing the time domain into subintervals, using a computationally cheap coarse propagator (CP) to generate an initial estimate, and then refining the solution in parallel within each subinterval using an accurate but expensive fine propagator (FP). The CP provides a rough, low-resolution approximation of the solution, while the FP solves the problem with high accuracy.  The parareal method has become a foundational tool to achieve more speedup particularly when parallelization  in space saturates.    
 It has been successfully applied to a wide range of problems, including option pricing \cite{BalMaday:2002,PagesPironneau:2016}, multiscale models \cite{Engblom:2009,ArielKim:2016}, stochastic models \cite{Bossuyt:2025,BrehierWang:2020}, nonlocal models \cite{XuHesthaven:2015,Li:2024}, and optimal control problems \cite{Fang:2022,GanderKwok:2020} etc. 
 
 
 The convergence analysis of the parareal algorithm for linear PDEs, particularly linear parabolic equations, has been extensively studied \cite{GanderVandewalle:2007,WuZhou:2015,wu2015convergence2,Dobrev:2017,hessenthaler2020multilevel,legoll2013micro}. When the initial condition has only limited regularity, existing results indicate that the behavior of parareal becomes worse \cite{dai2013stable} and may converge linearly with the iteration number, where the linear convergence factor highly relies on the choice of CP and FP \cite{WuZhou:2015,jin2025optimizing}. This observation has motivated recent efforts to optimize the CP, either based on the specific model \cite{FriedhoffSouthworth:2021,DeSterck:2021} or relative to a fixed FP \cite{jin2025optimizing}. 
 While the convergence frameworks for linear problems have shed valuable insights into the algorithm's strengths and limitations, a fully rigorous convergence analysis for nonlinear parabolic problems remains unavailable. There are two primary strategies for deriving the convergence results for nonlinear ordinary differential equations (ODEs). The first approach relies on assumptions regarding the stability and the order condition of the CP~\cite{gander2008nonlinear,pentland2023error,gander2019new}. Under these assumptions, the parareal error is decomposed into a stability error from the current iteration and an approximation error from the previous iteration, allowing them to be estimated independently.
 This approach has been extended to nonlinear parabolic problems under strong regularity conditions on the exact solution \cite{bal2005convergence,dai2013stable}. However, the decoupled treatment significantly loosens the error bound when applied to nonlinear parabolic problems with limited regularity data, yielding an estimate that fails to recover the sharp linear bound when the nonlinear term vanishes. The second approach depends on the one-sided Lipschitz assumption on the nonlinear function \cite{gander2013analysis}. Rather than bounding the difference between the CPs and the FPs from the previous iteration directly, this method separately estimates the deviations of the CPs and FPs, which amplifies the error bound, particularly when the nonlinearity is weak or absent. To the best of our knowledge, the bounds for nonlinear problems in these works cannot recover that for linear problems when the nonlinear term vanishes.
 In the special case where the CP is chosen as a linearized exponential integrator, the analysis becomes more tractable, and the convergence can be established with relative simplicity. Brehier and Wang \cite{BrehierWang:2020} demonstrated the convergence of the parareal method for solving semilinear stochastic PDEs under this choice of CP. However, their approach does not extend to standard linearized time-stepping schemes, which are more commonly employed as coarse propagators in practical applications. These challenges highlight the need for further investigation into the convergence of the parareal algorithm for semilinear parabolic problems.

 In this work, we study convergence properties of the parareal algorithm for  semilinear parabolic equations with initial data in the domain $ \Dom(A) $. Previous studies have reported that, for nonlinear problems, the convergence behavior of the parareal algorithm often resembles that of linear problems when the coarse propagator (CP) is sufficiently well-resolved \cite{jin2025optimizing}. However, to the best of our knowledge, no theoretical guarantees have been provided to substantiate the empirical observation. In Theorem \ref{thm:main_thm_L2}, we establish that when the CP has sufficient resolution, the convergence factor of the parareal iteration is given by $ \gamma_\text{lin} + {C_f} \Delta T |\ln \Delta T|$, where $ \gamma_\text{lin} $ is the convergence factor for the corresponding linear parabolic equation, $ \Delta T $ is the time step size of the CP, and {$ C_f $} is a constant that depends on the nonlinearity of the problem (e.g., boundedness and Lipschitz constant, etc.). In case that the source term $f$ is independent of $u$, the constant {$ C_f $} vanishes, and the convergence factor reduces to $ \gamma_\text{lin} $. Furthermore, when the CP achieves high accuracy, i.e., as $ \Delta T \to 0 $, the convergence factor also approaches $ \gamma_\text{lin} $. 
 This result is derived by combining the error-splitting framework from the superlinear convergence analysis of the parareal method \cite{GanderVandewalle:2007}, a refined linear convergence theory for linear parabolic equations \cite{GanderVandewalle:2007,WuZhou:2015}, and nonsmooth data error analysis \cite{thomee2007galerkin}. The analysis does not rely on any specific structure of the CPs and is more broadly applicable than existing approaches. The numerical experiments fully support the theoretical findings.
 
 The rest of the paper is organized as follows. In Section \ref{sec:parareal}, we describe the parareal algorithm. In Section \ref{sec:linear conv}, we prove the linear convergence. In Section \ref{sec:numerical}, we numerically illustrate the theoretical findings. Throughout,  $(\cdot,\cdot)$ denotes the $L^2(\Omega)$ inner product or duality pairing between $H^{-1}(\Omega)$ and $H_0^1(\Omega)$.   Let $\{(\lambda_j, \fy_j)\}_{j=1}^\infty$ denote the eigenpairs of $A$, where $\{\fy_j\}_{j=1}^\infty$ forms an orthonormal basis in $L^2(\Omega)$. 
 We denote by $\dH q$ the Hilbert space induced by the norm
 $\|v\|_{\dH q}^2:=\|A^\frac{q}{2}v\|_{L^2\II}^2=\sum_{j=1}^{\infty}\lambda_j^q ( v,\fy_j )^2$ for $ q\ge {-2}$. 
 For $q\in[-2,0)$, the negative norm is equivalent to the norm of the dual space of $\dot H^{-q}(\Omega)$.
 $\mathcal{L} (L^2 (\Omega ))$ denotes the space of linear bounded operators from $L^2 (\Omega) $ to $L^2 (\Omega )$, with the operator norm $\| \cdot \|_{\mathcal{L}(L^2 (\Omega))}$. 
 
 \section{Preliminaries}\label{sec:parareal}
 
 To discretize problem \eqref{eqn:pde} in time, we divide the time interval $(0,T)$ into $N$ equal
 subintervals, each of length ${\Delta t} = T/N$. Let $\Delta T = J \Delta t$ ($J\in \mathbb{N}$) be the coarse step size, $N_c = T/\Delta T\in \mathbb{N}$, and  $T_n = n\Delta T$. 
 We present the parareal algorithm for problem \eqref{eqn:pde} in Algorithm \ref{alg:para}. In practice, the CP $\mathcal{G}$ is often an inexpensive low-order  method, whereas the FP $\mathcal{F}$ is a high-order but expensive time integrator. Given any $v\in L^2\II$, the CP, denoted by $\mathcal{G}_{\Delta T}(t,v)$, evolves the initial state $v$ from time $t$ to $t+\Delta T$, and the FP is denoted by $\mathcal{F}_{\Delta T}(t,v)$.  For uniform time meshes, we can write  
 \begin{equation*} 
 	\mathcal{G}_{\Delta T} (t,v) =: \mathcal{G}(v)\quad \text{and}\quad \mathcal{F}_{\Delta T} (t,v) =: \mathcal{F}(v).
 \end{equation*}
 Let FP be an exact solver: for $v \in L^2\II$, $\mathcal{F}(v)$ is given by $\mathcal{F}(v) = w(\Delta T)$, with $w(\Delta T)$ solving problem \eqref{eqn:pde} with the initial condition $v$ at time $\Delta T$:
 \begin{equation}\label{eqn:fp}
 	\mathcal{F}(v) := w(\Delta T) = e^{- A \Delta T} v + \int_{0}^{\Delta T} e^{ -A(\Delta T-s) } f( w(s)) \, \d s.
 \end{equation}
 Throughout, the CP in the parareal algorithm is a single step scheme: for any $v \in  L^2\II$,
 \begin{equation}\label{eqn:cp}
 	\mathcal{G}( v) =R( \Delta TA) v + \Delta TP( \Delta TA) f(v),
 \end{equation}
 where $R(s)$ and $P(s)$ are rational functions approximating $e^{-s}$ and $s^{-1}(1 - e^{-s})$, respectively {\cite{thomee2007galerkin}}. 
 
 \begin{algorithm}[hbt!]
 	\center
 	\caption{Parareal algorithm.}
 	\begin{algorithmic}[1]\label{alg:para}
 		\STATE \textbf{Initialization}:
 		Compute
 		$U^0_{n+1} = \mathcal{G}_{\Delta T}(T_{n}, U^0_{n})$ with $U_0^0 =u_0$, $n =0,1,...,N_c - 1$.
 		\FOR {$k=0,1,\ldots,K$}
 		\STATE \textbf{Parfor}: On each subinterval $[T_n,T_{n+1}]$, sequentially compute the fine correction:
 		\begin{equation*}  
 			U_{n+1}^{k} = \mathcal{F}_{\Delta T}(T_n,  U_{n}^{k}). 
 		\end{equation*}
 		\STATE Perform sequential corrections, i.e., find $U^{k+1}_{n+1}$ by       \begin{equation}\label{eqn:parareal}
 			U^{k+1}_{n+1} = \mathcal{G}_{\Delta T}(T_{n}, U^{k+1}_{n})  +  U_{n+1}^{k} -  \mathcal{G}_{\Delta T}( T_n, U^{k}_{n}), 
 		\end{equation}
 		with $U^{k+1}_0 = u_0$, for $n =  0,1,...,N_c - 1$.\vskip5pt
 		\STATE Check the stopping criterion.
 		\ENDFOR
 	\end{algorithmic}
 \end{algorithm}
 
 We make the following assumption on the nonlinear function $f$ in problem \eqref{eqn:pde}. 
 The assumption can be relaxed to a strip along the exact solution $u$ in the 
 $H^2\II$ norm. 
 \begin{assumption}\label{assum:f}
 	The function {$f: \mathbb{R} \rightarrow \mathbb{R}$} is smooth and  bounded, with its derivatives up to the third order also bounded.
 \end{assumption}
 
 We have the following estimates on $f$. The proof is given in Section \ref{sec:proof2} in the appendix.
 \begin{lemma}\label{lem:f}
 	Let the function \( f \) satisfy Assumption \ref{assum:f}. Then the following estimates hold.
 	\begin{align}
 		\left\|( f'(u) - f'(v))(h) \right\|_{\dot{H}^{-2}(\Omega)} &\leq  C \| u - v \|_{L^2(\Omega)} \| h \|_{L^2(\Omega)}, && \forall u, v \in L^2(\Omega), \, h \in L^2(\Omega), \label{eqn:basic-est1} \\
 		\left\|( f'(u) - f'(v))(h) \right\|_{L^2(\Omega)} &\leq   C \| u - v \|_{L^2(\Omega)} \| h \|_{H^2(\Omega)}, && \forall u, v \in L^2(\Omega), \, h \in H^2(\Omega), \label{eqn:basic-est2} \\
 		\left\| f(u) - f(v) \right\|_{\dot{H}^1(\Omega)} &\leq  C \| u - v \|_{H^2(\Omega)}, && \forall u, v \in H^2(\Omega). \label{eqn:basic-est3}
 	\end{align}
 	Additionally, for \( u \in H^2(\Omega) \), the estimate $\left\| f'(u) \right\|_{H^2(\Omega)} \leq  C (1 + \| u \|_{H^2(\Omega)})$ holds. 
 \end{lemma}
 
 Under Assumption \ref{assum:f}, the solution $u(t)$ is uniformly bounded in the $\dH2$ norm, if $u_0\in \dH2$. 
 \begin{lemma}\label{lem:regu_U_n}
 	Let $u$ be the solution to  \eqref{eqn:pde} with $u_0 \in \dH2$. There holds $ \| u ( t  ) \|_{\dH2} \leq C_T\| u_{0}\|_{\dH2}$.
 \end{lemma}
 \begin{proof}
 	By the solution representation
 	$u(t) = e^{-At} u_0 + \int_0^t e^{-As} f(u(t-s)) \,\d s$, we have
 	\begin{equation*}
 		\partial_tu(t) = -e^{-At} Au_0 + {e^{-At}f(u_0)} + \int_0^t e^{-As}f'(u(t-s)) \partial_tu(t-s) \,\d s.
 	\end{equation*}
 	From Assumption \ref{assum:f}, we deduce 
 	\begin{equation*}
 		\|\partial_tu(t)\|_{L^2 \II } \le \|u_0\|_{\dH2} +C \int_0^t  \|\partial_tu (t-s)\|_{L^2 \II } \,\d s.
 	\end{equation*}
 	Then Gronwall's inequality implies 
 	$\|\partial_tu(t)\|_{L^2 \II } \le C_T\|u_0\|_{\dot{H}^2 \II }$  for all $t\in[0,T]$.
 	This and equation \eqref{eqn:pde} imply the desired result.
 \qed \end{proof}
 
 With the exact solver for the FP and the stability function $R$ of  the CP, we define 
 \begin{equation}\label{eqn:gamma}
 	\gamma(s)=\frac{e^{-s}-R(s)}{1-|R(s)|} \quad\mbox{and}\quad \gl =\sup_{s\in( 0,\infty)} |\gamma(s)|.
 \end{equation}
 
 \begin{remark}\label{rem:gamma_lin}
 	The convergence factor $\gl$ characterizes the convergence rate of the parareal algorithm for linear problems with the CP given in \eqref{eqn:cp}  \cite{WuZhou:2015,jin2025optimizing}. 
 	For the exact FP, there holds
 	\begin{equation}\label{eqn:gl}
 		\gl = 
 		\begin{cases} 
 			0.298, & \text{if CP = BE},~R(s)=\frac{1}{1+s}, \\
 			0.082, & \text{if CP = two-stage Lobatto IIIC},~R(s)=\frac{2}{2+2s+s^2}, \\
 			0.016, & \text{if CP = OCP \cite{jin2025optimizing}},~R(s)=\frac{1-0.17922s}{1+0.82078s+0.42444s^2}.
 		\end{cases}
 	\end{equation}
 	Here BE denotes the backward Euler propagator, and  OCP denotes the optimized coarse propagator 
 	for the exact FP, proposed in \cite{jin2025optimizing}.
 \end{remark} 
 
 We impose the following assumption on $R$ and $P$ in \eqref{eqn:cp}. 
 \begin{assumption}\label{assum:R}
 	The rational functions $R$ and $P$ in  \eqref{eqn:cp}
 	satisfy the following conditions:
 	\begin{enumerate}
 		\item[{\rm(i)}] $R(0)=1$ and $R'(0)=-1$. 
 		$R(x) \in (-1,1)$ and $|xR(x)|\leq C$ for all  $x\in (0,\infty]$.
 		\item[{\rm(ii)}] 
 		$P(0)=1$. $|P(x)|$ and $|xP(x)|\leq C$ for $x\in (0,\infty]$.
 		\item[{\rm(iii)}] Let $Q\left( s,r \right) =\left( R\left( s \right) \right)^{r-1} \left( e^{-s}-R\left( s \right) \right)$. For $r\in[1,\infty)$, let $\Lambda_{r} =\left\{ s>0,\partial_{s} Q\left( s,r \right) =0 \right\}$ and
 		\begin{equation}\label{hr}
 			h(r) =\min_{s\in \Lambda_{r}} |R\left( s \right) -\left( 1-r^{-1} \right) |.
 		\end{equation}
 		The function $h$ has a finite number of simple zeros, and $h(1)\neq0$.
 		\item[{\rm(iv)}] Let $\{s_k^\ast\}_{k=1}^N$ be zeros of $\gamma'(s)$ with $R(s)> 0$. The function $|\gamma(s)|$ attains its unique supremum at $s_0 \in \{s_k^\ast \}_{k=1}^N$ and $e^{-s_0}+R'(s_0)\neq 0$.
 	\end{enumerate}
 \end{assumption}
 
 \begin{remark}\label{rem:ass}
 	In Assumption \ref{assum:R}, conditions (i) and (ii) indicate that the CP \eqref{eqn:cp} is a first-order solver, while conditions (iii) and (iv) are employed to establish Theorem \ref{thm:inv_gamma_lin}. Fig.~\ref{fig:assum} illustrates the function $h(r)$ in (iii) for three CPs in Remark \ref{rem:gamma_lin}. From the derivative $\partial_s Q(s,r)$ in \eqref{eqn:Q,gamma} (in the appendix), the conditions $\partial_s Q(s,r) = 0$ and $\gamma'(s)=0$ imply $h(r)=0$. That is, the extremal points of $|\gamma(s)|$ are the zeros of $h(r)$. Let  $\{r_k^\ast\}_{k=1}^K$ are zeros of $h(r)$. Under condition (iii), they are simple roots, and hence there holds $h'(r_k^*\pm)\neq 0$. Therefore, for any $r_0 > 1$, there exist $\delta >0$ and $a>0$ such that 
 	\begin{equation}\label{eqn:f_lower bound}
 		h(r) \geq \sum_{k=1}^K \chi_\delta^{(k)} (r)\cdot a|r-r_k^\ast|,\quad \forall r\in \bigcup_{k=1}^K [r_k^* - \delta, r_k^* + \delta],
 	\end{equation}
 	where $\chi_\delta^{(k)}$ is the characteristic function of the interval $[r_k^\ast -\delta , r_k^\ast + \delta]$. The dotted lines in Fig.~\ref{fig:assum} denote the lower bounds. The inequality \eqref{eqn:f_lower bound} is used in Step (iii) of the proof of Theorem \ref{thm:inv_gamma_lin} to locally bound $R(s)$ and $1-R(s)$ using linear approximations. Condition (iv) is used in Step (ii) to ensure that  $h(r)\neq 0$  implies $ \gamma'(s) \neq 0 $ when $\partial_s Q(s,r)=0$ in \eqref{eqn:Q,gamma}. The condition $R(s) > 0$ ensures that $\gamma'(s)$ is well-defined. These conditions can be readily verified for specific CPs.
 	
 	\begin{figure}[htbp!]
 		\centering
 		\begin{tabular}{ccc}
 			\includegraphics[width=0.3\textwidth]{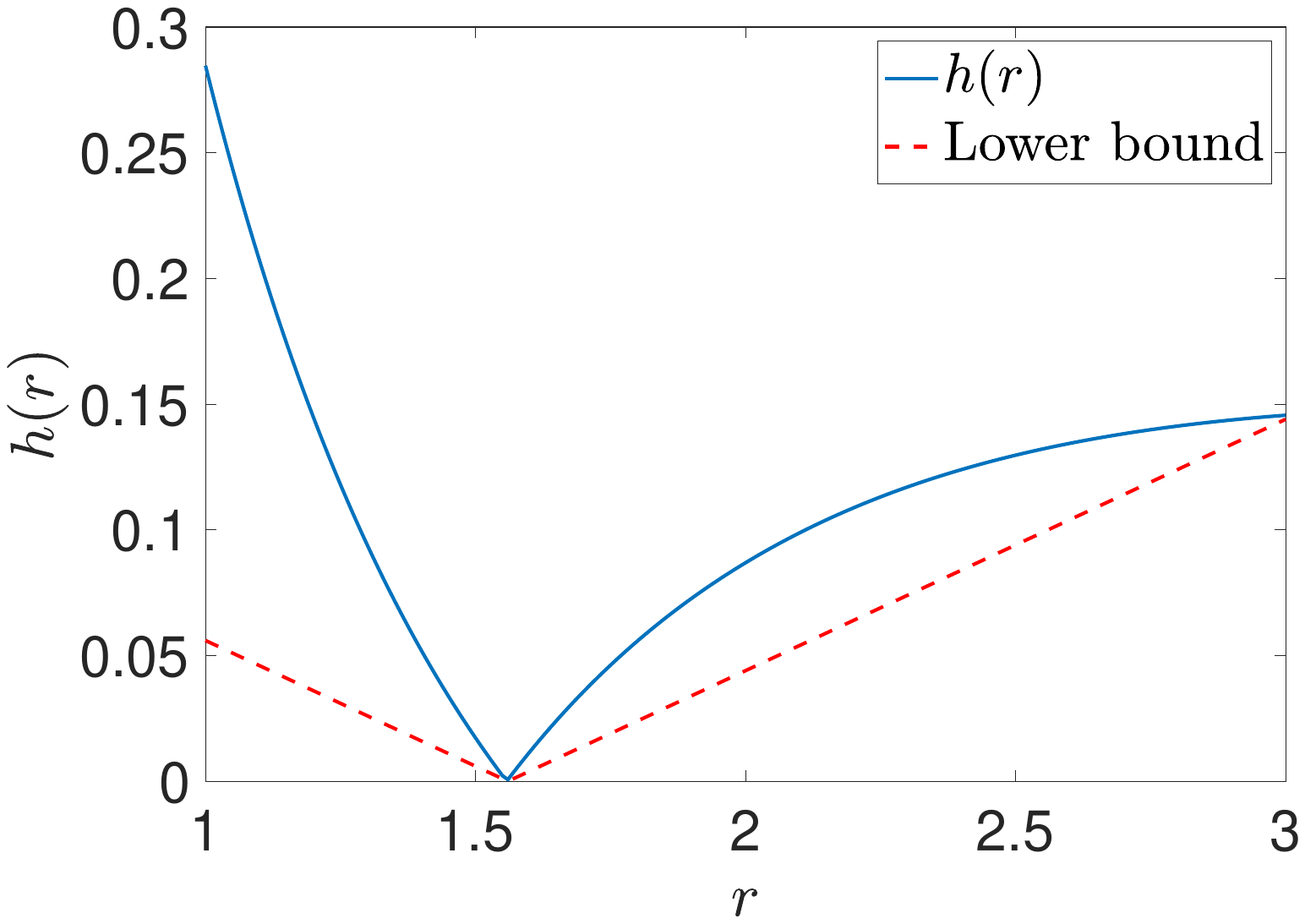}
 			&       \includegraphics[width=0.3\textwidth]{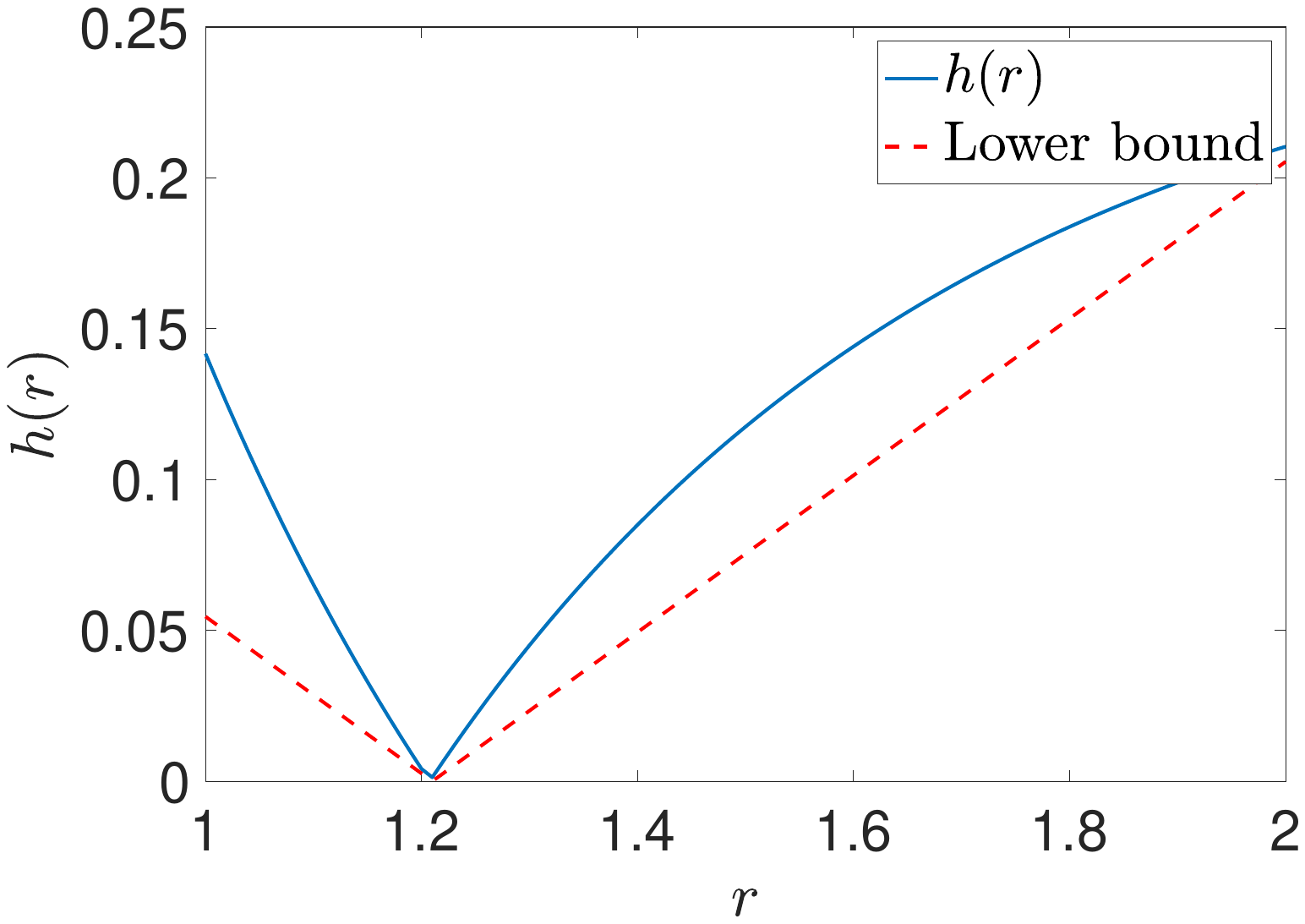}
 			&      \includegraphics[width=0.3\textwidth]{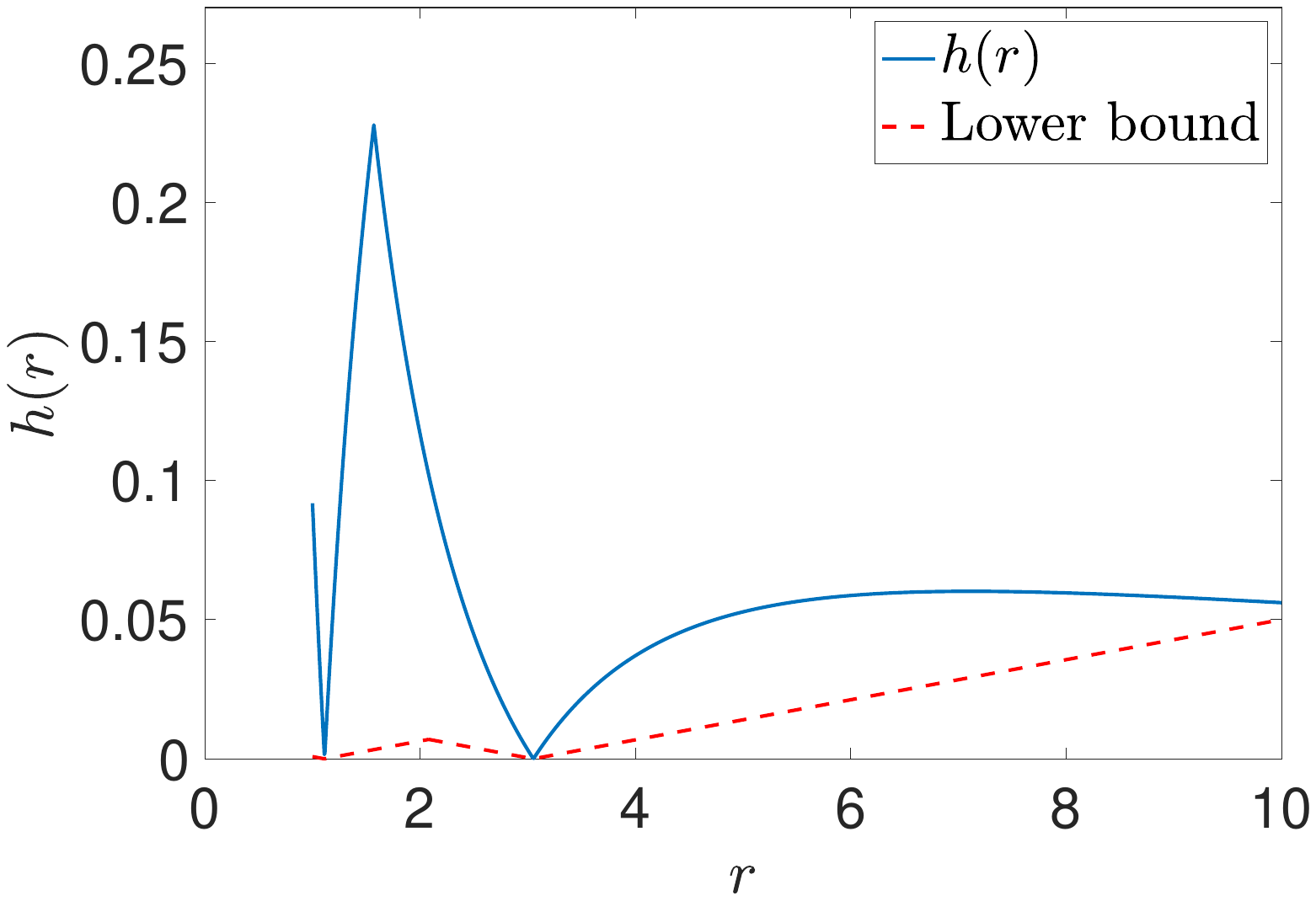}\\
 			\includegraphics[width=0.3\textwidth]{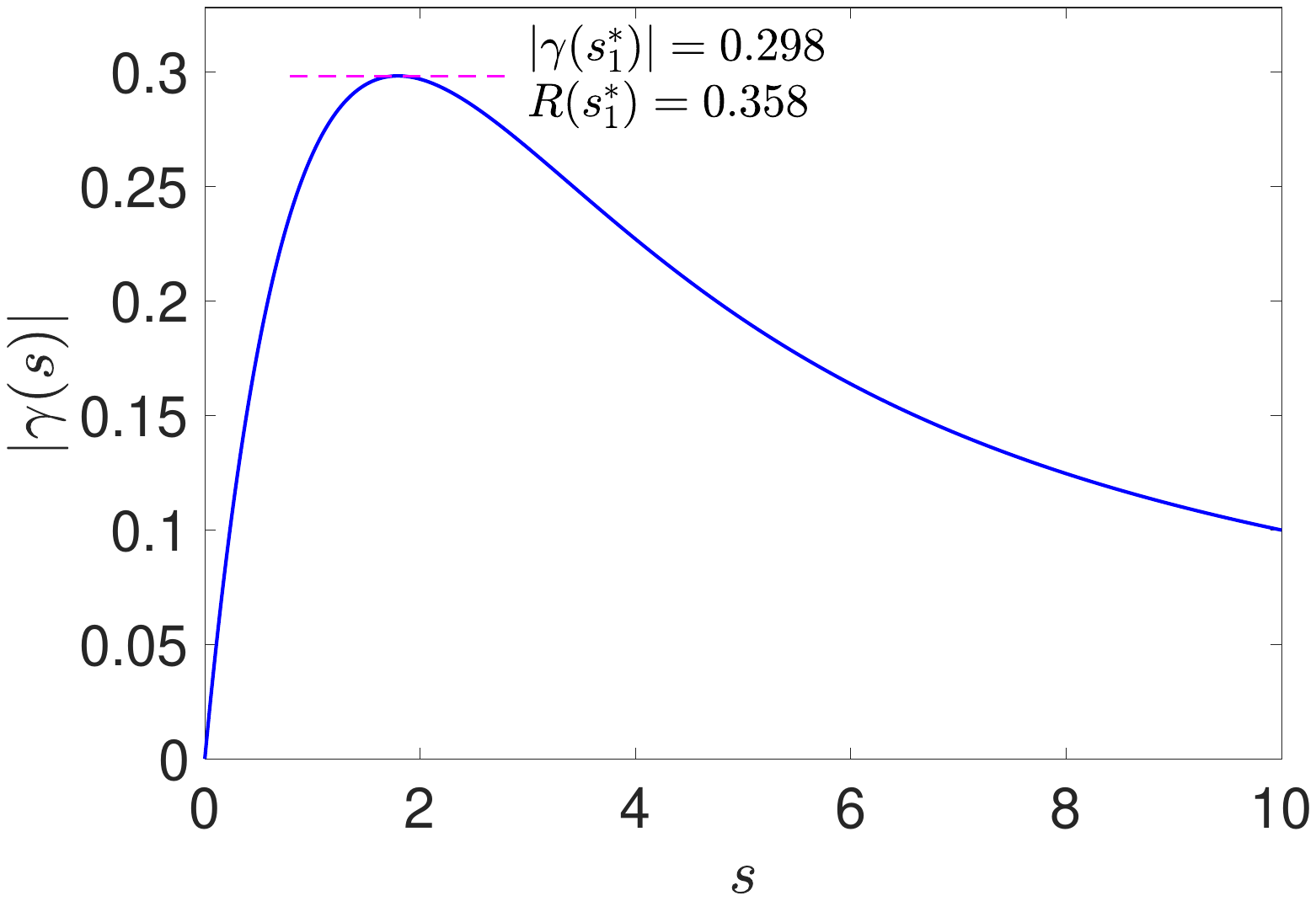}
 			&    \includegraphics[width=0.3\textwidth]{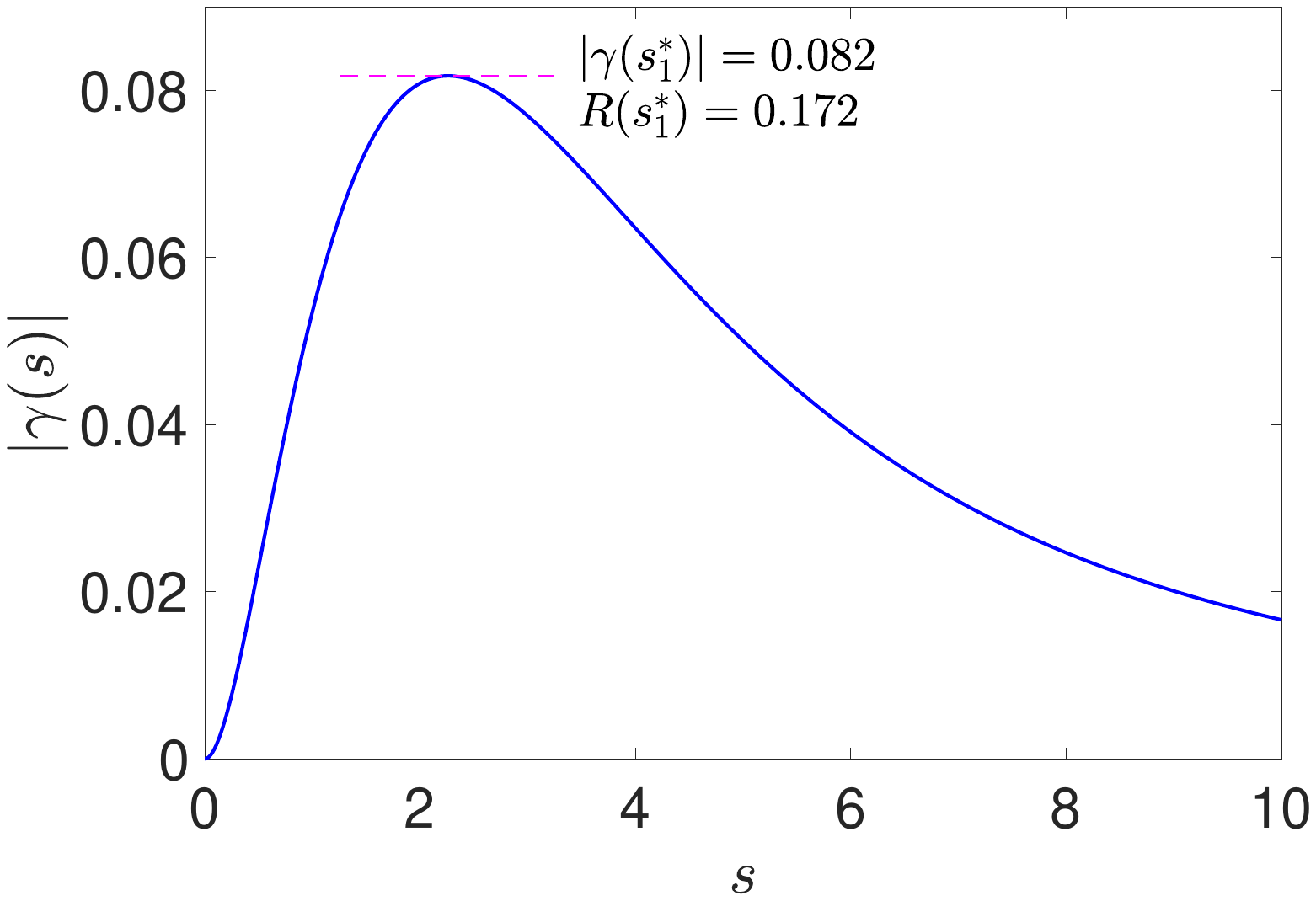}
 			&    \includegraphics[width=0.3\textwidth]{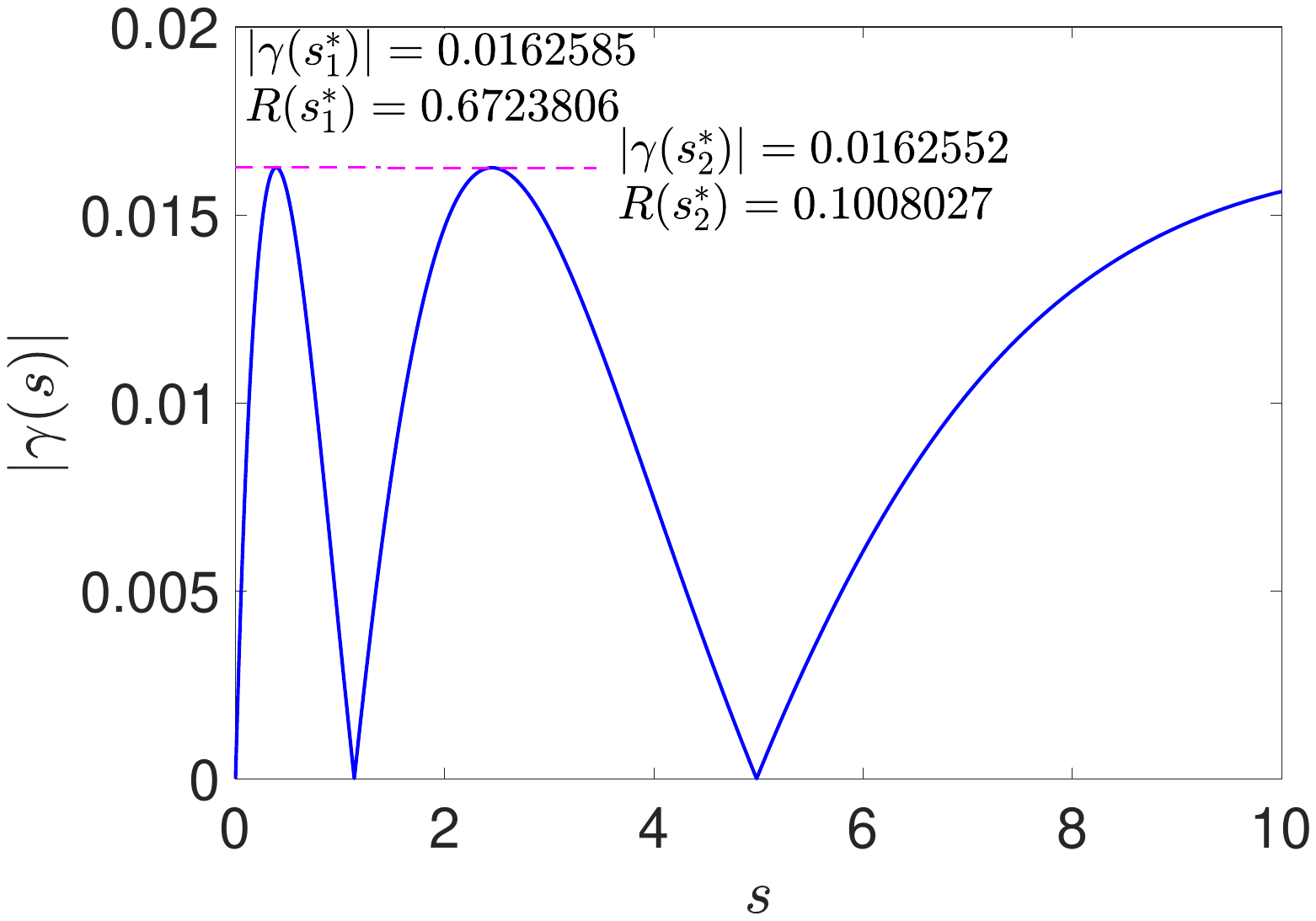}\\
 			(a) BE & (b) Two-stage Lobatto IIIC  & (c) OCP
 		\end{tabular}    
 		\caption{The plots  of $h(r)$ and $|\gamma(s)|$ when CP is BE, two-stage Lobatto IIIC method, and OCP.}
 		\label{fig:assum}
 	\end{figure}
 \end{remark}

 \section{Linear convergence of the parareal algorithm}
 \label{sec:linear conv}
 In this section, we establish the linear convergence of the parareal algorithm for problem \eqref{eqn:pde}. 
 We define the operators
 \begin{equation} \label{eqn:al-beta}
 	\alpha = R(\Delta TA)\quad \mbox{and}\quad  \beta = e^{-\Delta TA} - R(\Delta TA),
 \end{equation}
 and define $\text{Res}(u,v)$ by
 \begin{equation}\label{eqn:def_Res}
 	\begin{aligned}
 		(\mathcal{F} - \mathcal{G})(u) - (\mathcal{F} - \mathcal{G})(v) &= (e^{ -\Delta TA} - R(\Delta TA))(u-v) + \text{Res}(u,v) \\
 		&= \beta(u - v) + \text{Res}(u,v).  
 	\end{aligned}
 \end{equation}
 Let $E_n^k: = U_{n}^k-U_n$ be the parareal error. In the analysis of linear PDEs, Gander and Vanderwalle \cite{GanderVandewalle:2007} derived an explicit relationship between $E_{n+1}^{k+1}$ and the initial errors $E_i^0$, $i=0,\cdots,n-k$. Inspired by the analysis in \cite{GanderVandewalle:2007}, our key analysis strategy lies in extracting the linear part from the parareal iteration. The next lemma gives one crucial recursion for $E_{n+1}^{k+1}$.
 
 \begin{lemma}\label{lem:expression_para}
 	With the FP in \eqref{eqn:fp} and the CP in \eqref{eqn:cp}, the parareal error $E_n^k$ satisfies
 	\begin{equation}
 		\begin{aligned}
 			E_{n+1}^{k+1} &= \sum_{j=0}^{k} \sum_{i=j}^{n-1} \binom{i}{j} \alpha^{i-j} \beta^{j} \Delta TP(\Delta TA)(f( U_{n-i}^{k+1-j})-f( U_{n-i})) \label{eqn:para_whole}\\
 			&\quad + {\beta^{k+1} \sum_{i=k}^{n} \binom{i}{k} \alpha^{i-k} E_{n-i}^{0}} + \sum_{j=0}^{k} \sum_{i=j}^{n-1} \binom{i}{j} \alpha^{i-j} \beta^{j} ~\mathrm{Res}(U_{n-i}^{k-j}, U_{n-i}). 
 		\end{aligned}
 	\end{equation}
 \end{lemma}
 \begin{proof}
 	The parareal algorithm and the exact solution $U_{n+1}$ respectively satisfy
 	$$
 	U_{n+1}^{k+1}=\mathcal{G}( U_{n}^{k+1} ) +\mathcal{F}(U_{n}^{k}) -\mathcal{G}( U_{n}^{k})\quad \mbox{and}\quad U_{n+1}=\mathcal{G}(U_{n}) +\mathcal{F}(U_{n}) -\mathcal{G}(U_{n}).
 	$$ Thus the error $E_{n+1}^{k+1}=U_{n+1}^{k+1}-U_{n+1}$ satisfies
 	\begin{align*}
 		E_{n+1}^{k+1} &=(\mathcal{G}( U_{n}^{k+1}) -\mathcal{G}( U_{n} ) ) + (( \mathcal{F}( U_{n}^{k}) - \mathcal{G}(U_{n}^{k})) - ( \mathcal{F}( U_{n}) - \mathcal{G} (U_{n}))) \\
 		&= \alpha E_{n}^{k+1} + \Delta TP( \Delta TA)( f( U_{n}^{k+1}) - f( U_{n})) \\
 		&\quad + ( \mathcal{F}(U_{n}^{k}) -\mathcal{G}(U_{n}^{k})) - (\mathcal{F}( U_{n}) - \mathcal{G}(U_{n})) \\
 		&= \alpha^{n+1} E_0^{k+1} + \sum_{i=0}^{n-1} \alpha^{i} \Delta TP( \Delta TA)(f(U_{n-i}^{k+1}) - f( U_{n-i})) \\
 		&\quad + \sum_{i=0}^{n-1} \alpha^{i} (( \mathcal{F} -\mathcal{G})( U_{n-i}^{k}) -( \mathcal{F} - \mathcal{G}) ( U_{n-i})),
 	\end{align*}
 	with $E_0^{k+1} = 0$. 
 	Then the error $E_{n+1}^{k+1}$ satisfies
 	\begin{align*}
 		E_{n+1}^{k+1} &= \sum_{i=0}^{n-1} \alpha^{i} \Delta T P( \Delta TA )( f( U_{n-i}^{k+1}) - f( U_{n-i}))  + \sum_{i=0}^{n-1} \alpha^{i} \beta E_{n-i}^{k} + \sum_{i=0}^{n-1} \alpha^{i} ~\text{Res}( U_{n-i}^{k}, U_{n-i}).
 	\end{align*}
 	Next we expand the linear part $\sum_{i=0}^{n-1} \alpha^{i} \beta E_{n-i}^{k}$ using the parareal iteration:
 	\begin{align*}
 		\sum_{i=0}^{n-1} \alpha^{i} \beta E_{n-i}^{k} &= \sum_{i=0}^{n-1} \sum_{j=0}^{n-i-1} \alpha^{i+j} \beta \Delta TP( \Delta TA)( f( U_{n-i-1-j}^{k}) - f( U_{n-i-1-j})) \\
 		&\quad+ \sum_{i=0}^{n-1} \sum_{j=0}^{n-i-1} \alpha^{i+j} \beta^{2} E_{n-i-1-j}^{k-1} + \sum_{i=0}^{n-1} \sum_{j=0}^{n-i-1} \alpha^{i+j} \beta~ \text{Res}( U_{n-i-1-j}^{k-1}, U_{n-i-1-j}) \\
 		&= \beta \sum_{i=0}^{n-1} \binom{i+1}{1}\alpha^{i} \Delta TP(\Delta TA)( f( U_{n-1-i}^{k}) - f(U_{n-1-i})) \\
 		&\quad + {\beta^{2} \sum_{i=0}^{n-1} \binom{i+1}{1}\alpha^{i} E_{n-1-i}^{k-1}} + \beta \sum_{i=0}^{n-1} \binom{i+1}{1} \alpha^{i} ~\text{Res}( U_{n-1-i}^{k-1}, U_{n-1-i} ).
 	\end{align*}
 	This identity connects $E_{n+1}^{k+1}$ with the errors at the $k-1,k$ and $k+1$th iterations:
 	\begin{align*}
 		E_{n+1}^{k+1} &= \sum_{i=0}^{n-1} \alpha^{i} \Delta TP( \Delta TA)( f( U_{n-i}^{k+1})-f( U_{n-i}))  + \sum_{i=0}^{n-1} \alpha^{i}~\text{Res}( U_{n-i}^{k}, U_{n-i})\\
 		&\quad+ \beta \sum_{i=1}^{n} \binom{i}{1} \alpha^{i-1} \Delta TP( \Delta TA) ( f( U_{n-i}^{k}) - f( U_{n-i})) \\
 		&\quad + {\beta^{2} \sum_{i=1}^{n} \binom{i}{1} \alpha^{i-1} E_{n-i}^{k-1}} + \beta \sum_{i=1}^{n} \binom{i}{1} \alpha^{i-1}~\text{Res}( U_{n-i}^{k-1}, U_{n-i}).
 	\end{align*}
 	Repeating the procedure $k-1$ times gives the desired identity.
 \qed \end{proof}
 
 Lemma \ref{lem:expression_para} provides the key recursion for the convergence analysis. We shall analyze each summation in \eqref{eqn:para_whole} separately in several technical lemmas:  Lemma~\ref{lem:lin} handles the linear part, and Lemmas~\ref{lem:f-f} and \ref{lem:Res} analyze the first summation and the last summation, respectively. 
 
 \subsection{Preliminary technical estimates}\label{subsec:prelim}
 Now we derive several preliminary estimates to bound the summations in Lemma \ref{lem:expression_para}. In particular, we establish the linear decay of the summation \begin{equation}\label{eqn:summation}
 	\sum_{i=j}^{\infty} \binom{i}{j}\sup_{s\in (0,\infty)}{ |R(s)|^{i-j}|e^{-s}-R(s)|^j},
 \end{equation}
 which enables bounding the first and third summations in \eqref{eqn:para_whole}. First, we estimate \eqref{eqn:summation} when $j\leq 2$. The bound \eqref{eqn:coro_ineq_2} involves the convergence factor $\gl$, but decays slower with respect to $i$.
 
 \begin{lemma}\label{lem:s R^i Q}
 	Let $R$ and $P$ satisfy Assumption \ref{assum:R}. Then for $\kappa \in [0,1]$ and $k,j\in \{0,1\}$,
 	there exists $C>0$ independent of $\kappa,k$ and $j$ such that for all $i\geq 1$,
 	\begin{equation}\label{eqn:coro_ineq}
 		\sup_{s\in ( 0,\infty)} \left|s^{\kappa}(R(s))^{i} \big( e^{-s}-R(s) \big)^{j} \Big(\frac{1-e^{-s}}{s} -P(s) \Big)^{k}\right|\leq \frac{C}{( i+1)^{2j+k+\kappa}}.
 	\end{equation}    
 	Moreover, for $\kappa \in [0,1]$ and $k\in \{0,1\}$,
 	there exists $C>0$ independent of $\kappa$ and $k$ such that for all $i\geq 1$,
 	\begin{equation}\label{eqn:coro_ineq_2}
 		\sup_{s\in ( 0,\infty)} \left|s^{\kappa}(R(s))^{i} \big( e^{-s}-R(s) \big)\Big(\frac{1-e^{-s}}{s} -P(s) \Big)^{k}\right|\leq \frac{C\gl}{( i+1)^{1+k+\kappa}}.
 	\end{equation}
 \end{lemma}
 \begin{proof}
 	Let
 	$Q(s) = s^{\kappa} ( e^{-s} - R(s) )^{j} ( \frac{1 - e^{-s}}{s} - P(s))^{k}$. Since \(|s R(s)| \leq C\) and \(|s P(s)| \leq C\) for \(s \geq 0\), we can choose \(S > 0\) sufficiently large and \(q < 1\) such that \(|R(s)| \leq C/s < q\) and \(|s^{\kappa} R(s)| \leq C\) for all \(s \in (S, \infty)\). Then we have
 	\begin{equation*}
    \begin{split}
 		\sup_{s \in (S, \infty)} \left| (R(s))^i Q(s) \right|
 		&\leq  \sup_{s \in (S, \infty)} |R(s)|^{i-1} \cdot  \sup_{s \in (S, \infty)} \frac{|Q(s)|}{s^{\kappa}} \cdot  \sup_{s \in (S, \infty)} |s^{\kappa} R(s)| \\
        &\leq C q^{i-1} \leq \frac{C }{(i+1)^{2j + k + \kappa}},
    \end{split}
 	\end{equation*}
 	since the exponential decay is dominated by the polynomial decay. The order conditions on $R$ and $P$ imply $\frac{{\rm d}^{q}}{{\rm d}s^{q}} \frac{Q\left( s \right)}{s^{\kappa}}=0$ for $q=0,\cdots,2j+k-1$, and thus $|\frac{Q(s)}{s^\kappa}| \leq C s^{2j+k}$ for $s\in (0,S)$. Further, for $s\in (0,S)$, we have $|R(s)|\leq e^{-cs}$ for some $c>0$ \cite[Lemma 9.2]{thomee2007galerkin}. Then, for $i \geq 1$, we derive
 	\begin{equation*}
 		\sup_{s\in(0,S)} |(R(s))^i Q(s)| \leq C\sup_{s\in (0,S)} |e^{-ics} s^{2j+k+\kappa}| = \frac{1}{(ci)^{2j+k+\kappa}} \sup_{s\in (0,S)} |e^{-ics} (ics)^{2j+k+\kappa}|.
 	\end{equation*}
 	The estimate $\sup_{s\in(0,\infty)} e^{-s}s^r=(\frac{r}{e})^r$ implies $\sup_{s\in (0,R)}|(R(s))^i Q(s)|\leq \frac{C}{( i+1)^{2j+k+\kappa}}$.
 	Combining the two cases proves \eqref{eqn:coro_ineq}. Next, by the definition of $\gl$ in \eqref{eqn:gamma}, we have 
 	\begin{align*}
 		&\sup_{s\in ( 0,\infty)} \left|s^{\kappa}(R(s))^{i} \big( e^{-s}-R(s) \big)\Big(\frac{1-e^{-s}}{s} -P(s) \Big)^{k}\right| \\
 		\leq &\gl \sup_{s\in ( 0,\infty)} \left|s^{\kappa}(R(s))^{i} \big( 1-|R(s)| \big)\Big(\frac{1-e^{-s}}{s} -P(s) \Big)^{k}\right|.
 	\end{align*}
 	The rest of the proof follows similarly as \eqref{eqn:coro_ineq}.
 \qed \end{proof}
 
 Since Lemma \ref{lem:s R^i Q} contains the factor $j^j$, it is useful only if $j$ is small. Next, we treat large $j$. Lemma \ref{lem:gamma_lin} bounds a series involving $R(s)$ and $e^{-s}$, whose growth is controlled by $\gl$.
 \begin{lemma}\label{lem:gamma_lin}
 	Let $R$ satisfy Assumption \ref{assum:R}. Then for $j\geq 0$ and $s\in (0,\infty)$, there holds
 	\begin{equation*}
 		\sum_{i=j}^{\infty} \binom{i}{j} |R( s) |^{i-j}|e^{-s}-R(s) |^{j+1}\leq \gl^{j+1}.
 	\end{equation*}
 \end{lemma}
 \begin{proof}
 	Note that $|R(s)| < 1$ for all $s \in (0, \infty)$. By the binomial series identity, we have  
 	\begin{equation*}
 		\sum_{i=j}^{\infty} \binom{i}{j} |R(s)|^{i-j} = \sum_{m=0}^{\infty} \binom{m+j}{j} |R(s)|^{m} = \left( \frac{1}{1 - |R(s)|} \right)^{j+1}.
 	\end{equation*}
 	Thus, by the definition \eqref{eqn:gamma} of the factor $\gl$, for all $s \in (0, \infty)$, we have
 	\begin{equation*}
 		\sum_{i=j}^{\infty} \binom{i}{j} |R(s)|^{i-j} |e^{-s} - R(s)|^{j+1} \leq \left( \frac{|e^{-s} - R(s)|}{1 - |R(s)|} \right)^{j+1} \leq \gl^{j+1}.
 	\end{equation*}
 \qed \end{proof}
 
 Next we refine  Lemma \ref{lem:gamma_lin} by taking an additional supremum over $s \in (0, \infty)$. See Appendix \ref{sec:proof1} for the lengthy and technical proof.
 \begin{theorem}\label{thm:inv_gamma_lin}
 	Let $R$ satisfy Assumption \ref{assum:R}. Then for $j \geq 2$, there exists $C>0$ independent of $j$ such that  
 	\begin{align*}
 		\sum_{i=j}^{\infty} \binom{i}{j} \sup_{s \in (0, \infty)} | ( R(s) )^{i-j} ( e^{-s} - R(s) )^{j} | &\leq C \gl ^{j}.
 	\end{align*}
 \end{theorem}
 
 To numerically verify Theorem \ref{thm:inv_gamma_lin}, let 
 $S(j) = \sum_{i=j}^{100} \binom{i}{j} \sup_{s \in (0, \infty)} |R(s)|^{i-j} |e^{-s} - R(s)|^{j}$.
 The plot of $S(j)$ for three different stability functions in Fig.~\ref{fig:S(j)} agrees with the prediction by Theorem \ref{thm:inv_gamma_lin}. 
 \begin{figure}[htbp]
 	\centering\setlength{\tabcolsep}{0pt}
 	\begin{tabular}{ccc}
 		\includegraphics[width=.32\textwidth]{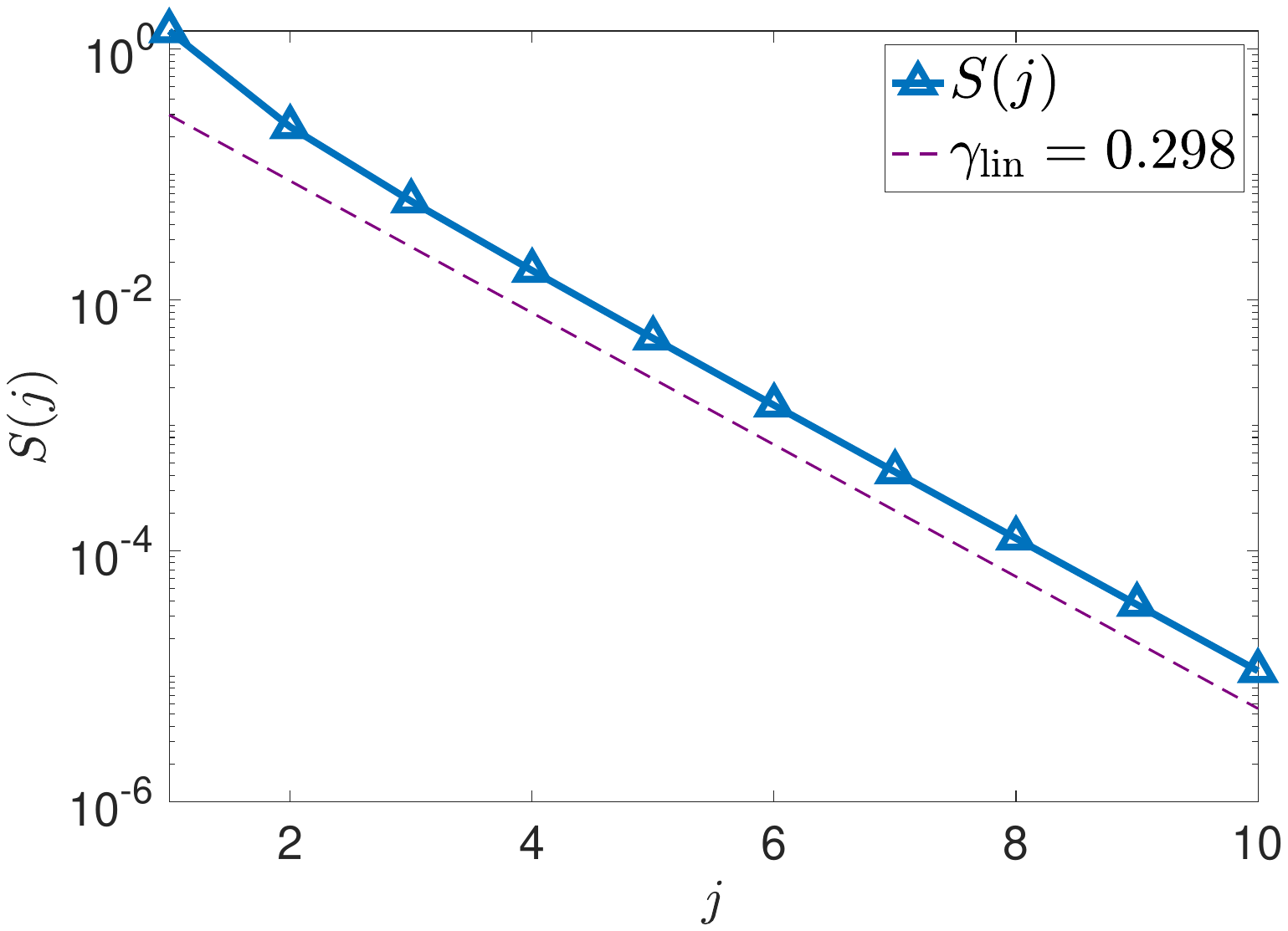}
 		& \includegraphics[width=.32\textwidth]{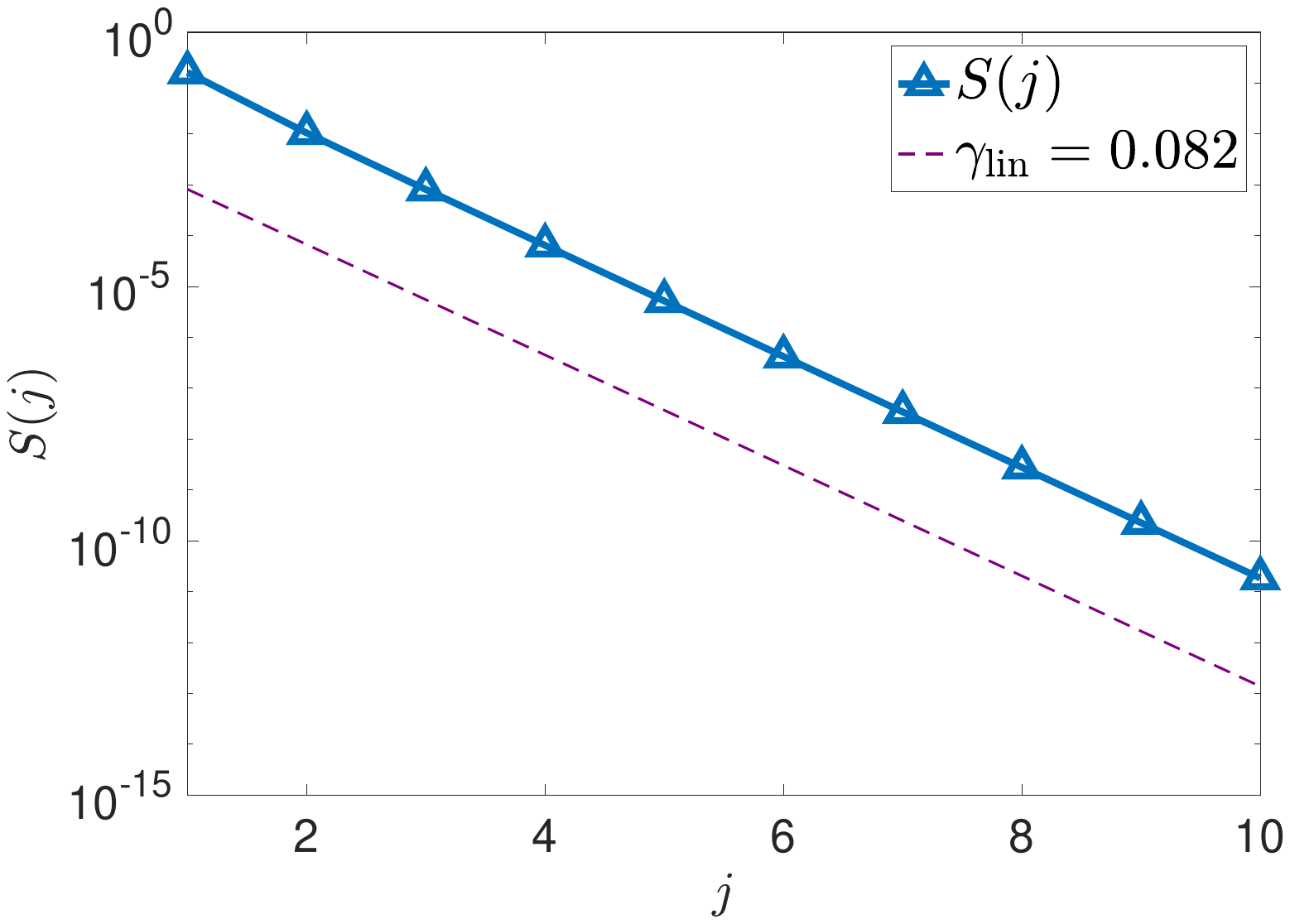}
 		&\includegraphics[width=.32\textwidth]{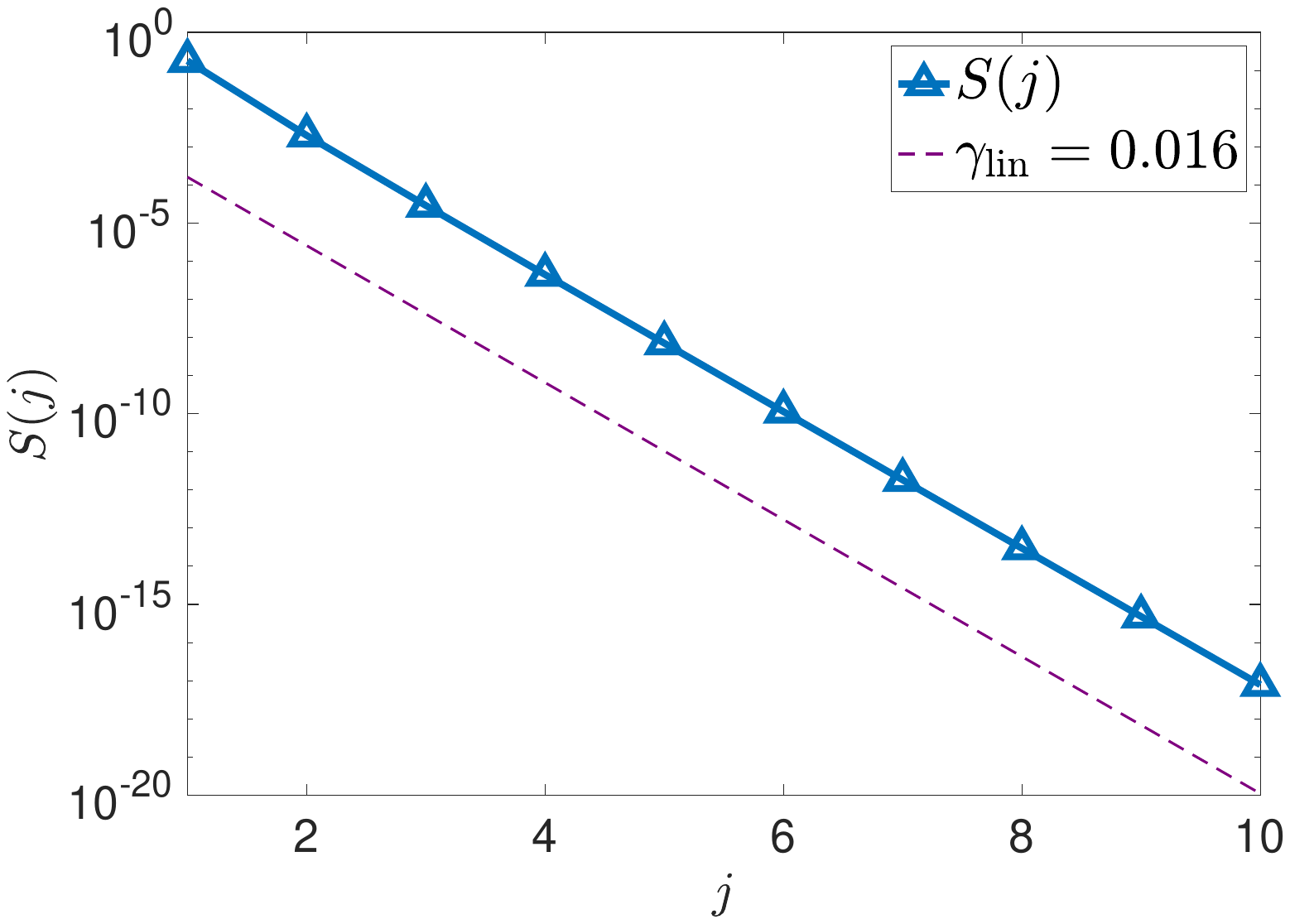}\\
 		(a) BE & (b) Two-stage Lobatto IIIC & (c) OCP
 	\end{tabular}
 	\caption{The graph of $S(j)$ for three CPs: BE, two-stage Lobatto IIIC and OCP.}
 	\label{fig:S(j)}
 \end{figure}
 
 \subsection{Error estimates}
 Next, we establish the linear convergence of the numerical solution by the parareal scheme \eqref{eqn:fp}. First, we present an estimate on the second summation in \eqref{eqn:para_whole}.
 \begin{lemma}\label{lem:lin}
 	Let the operators $\alpha$ and $\beta$ be defined in \eqref{eqn:al-beta}. Then the following estimate holds
 	\begin{align*}
 		\bigg\| \beta^{k+1} \sum_{i=k}^{n} \binom{i}{k} \alpha^{i-k} E_{n-i}^{0}\bigg\|_{\dH q} &\leq \gl^{k+1} \sum_{i=0}^{n} \| E_{i}^0\|_{\dH q},\quad q=0,2.
 	\end{align*}
 \end{lemma}
 
 \begin{proof}  
 	Recall that $\{(\lambda_j, \varphi_j)\}_{j\geq1}$ are the eigenpairs of the self-adjoint operator $A$. Then the definitions of the operators $\alpha$ and $\beta$ in \eqref{eqn:al-beta} and Lemma \ref{lem:gamma_lin} lead to
 	\begin{align*}
 		&\Big|\Big( \beta^{k+1} \sum_{i=k}^{n} \binom{i}{k} \alpha^{i-k} E_{n-i}^{0}, \varphi_j \Big)\Big| = | e^{-\Delta T \varphi_j} - R(\Delta T \lambda_j)|^{k+1} \Big| \sum_{i=k}^{n} \binom{i}{k} (R(\Delta T \lambda_j))^{i-k} \Big| |(E_{n-i}^{0},\varphi_j)| \\
 		\le& \max_{0\leq i\leq n}|(E_{n-i}^0, \varphi_j) || e^{-\Delta T \lambda_j} - R(\Delta T \lambda_j) |^{k+1}  \sum_{i=k}^{n} \binom{i}{k} |R(\Delta T \lambda_j)|^{i-k}
 		\le \gl^{k+1} \max_{0\leq i\leq n}|(E_{n-i}^0, \varphi_j)|.
 	\end{align*}
 	Then we obtain the desired estimates for $q=0,2$,
 	\begin{align*}
 		\bigg\| \beta^{k+1} \sum_{i=k}^{n} \binom{i}{k} \alpha^{i-k} E_{n-i}^{0}\bigg\|_{\dot H^q (\Omega)}^2
 		& \le \gl^{2(k+1)} \sum_{j=1}^{\infty} \lambda_j^{2q}\max_{0 \le i \le n} |( E_{n-i}^{0}, \varphi_j)|^{2} \\
 		&\le \gl^{2(k+1)} \sum_{i=0}^{n}    \sum_{j=1}^{\infty} \lambda_j^{2q}  |( E_{i}^{0}, \varphi_j)|^{2} \le \gl^{2(k+1)} \sum_{i=0}^{n} \| E_{i}^{0}\|_{\dot H^q (\Omega)}^{2}.
 	\end{align*}
 \qed \end{proof}
 
 The next lemma provides an estimate on the first summation  in \eqref{eqn:para_whole}. 
 \begin{lemma}\label{lem:f-f}
 	Let 
 	${\rm I}_j= \sum_{i=j}^{n-1} \binom{i}{j} \alpha^{i-j} \beta^{j} \Delta TP( \Delta TA ) ( f( U_{n-i}^{k+1-j} ) -f( U_{n-i}))$, and for $q=0,2$, $e_{n,q}^k=\max_{0\leq i\leq n}\|E_i^k\|_{\dH q}$. Then there exists  $C$ depending on $f$ but independent of $n$ and $k$ such that 
 	\begin{align*}
 		\| {\rm I}_j\|_{L^{2}(\Omega)} 
 		\leq& \begin{cases}C\Delta T \sum_{i=1}^{n} \| E_{i}^{k+1}\|_{L^{2}(\Omega)} ,&j=0,\\ C \Delta T \ln ne_{n,0}^k ,&j=1,\\ C \Delta T \gl^j  e_{n,0}^{k+1-j}, &j\geq 2,\end{cases}\\
 		\| {\rm I}_j\|_{\dot{H}^{2}(\Omega)} 
 		\leq& \begin{cases}C \Delta T^{\frac{1}{2}}  \sum_{i=1}^{n} \frac{1}{\sqrt{i+1}} \| E_{n-i}^{k+1}\|_{\dot{H}^{2}(\Omega)} ,&j=0,\\C \Delta T^{\frac{1}{2}}  e_{n,2}^k ,&j=1,\\ C \Delta T^{\frac{1}{2}}\gl^j e_{n,2}^{k+1-j},&j\geq 2.\end{cases}
 	\end{align*}
 \end{lemma}
 \begin{proof}
 	We analyze ${\rm I}_0$, ${\rm I}_1$ and ${\rm I}_{j}$ ($j\geq 2$) separately. First we prove the $L^2(\Omega)$ estimate. By the Lipschitz continuity of $f$, we have
 	$$ \|f(U_{n-i}^{k+1-j})-f(U_{n-i})\|_{L^2(\Omega)}\leq C\|E_{n-i}^{k+1-j}\|_{L^2(\Omega)}.$$ 
 	The bound on ${\rm I}_0$ follows directly from the stability of $\alpha$ and $P$.
 	For the case $j=1$, by the stability of $P$,
 	\begin{align*}
 		\|{\rm I}_1\|_{L^2(\Omega)} &\leq C\Delta Te_{n,0}^k\sum_{i=1}^{n-1} \binom{i}{1}\left\|  \alpha^{i-1} \beta \right\|_{\mathcal{L} (L^2 (\Omega))}.
 	\end{align*}
 	Meanwhile, by the estimate \eqref{eqn:coro_ineq} (with $\kappa, k=0$ and $j=1$), we derive
 	\begin{align*}
 		\sum_{i=1}^{n-1} \binom{i}{1}\left\|  \alpha^{i-1} \beta \right\|_{\mathcal{L} (L^2 (\Omega))}&\leq  \sum_{i=1}^{n-1} \binom{i}{1} \sup_{s\in( 0,\infty)} |(  R(s))^{i-1} ( e^{-s}-R(s)) | \le C \sum_{i=1}^{n-1} \frac1i 
 		\le  C   \ln n .
 	\end{align*}
 	Thus the bound on $\|{\rm I}_1\|_{L^2(\Omega)}$ follows.
 	Similarly, for $j \geq 2$, Theorem \ref{thm:inv_gamma_lin} implies
 	\begin{align*}
 		\|{\rm I}_j \|_{L^{2}(\Omega)}
 		&\leq C\Delta Te _{n,0}^{k+1-j} \sum_{i=j}^{n-1} \binom{i}{j} \sup_{s\in \left( 0,\infty \right)} |\left( R\left( s \right) \right)^{i-j} \left( e^{-s}-R\left( s \right) \right)^{j} | \leq C \Delta T\gl^j e_{n,0}^{k+1-j}.
 	\end{align*}
 	Next we estimate the $\dot{H}^2(\Omega)$ norm. For the term ${\rm I}_0$, by Lemma \ref{lem:f}, we obtain
 	\begin{align*}
 		\left\|{\rm I}_0  \right\|_{\dot{H}^2(\Omega)} \leq & \Delta T^{\frac{1}{2}} \left\| \sum_{i=0}^{n-1} \alpha^{i} {( \Delta T A)^{\frac{1}{2}}} P( \Delta TA) ( A^{\frac{1}{2}}( f( U_{n-i}^{k+1} ) - f( U_{n-i} )) ) \right\|_{L^{2}(\Omega)} \\
 		\leq &C \Delta T^{\frac{1}{2}}  \sum_{i=0}^{n-1}  \| \alpha^{i} {(\Delta T A)^{\frac{1}{2}}} P( \Delta TA ) \|_{\mathcal{L}(L^2 \II )} \|E_{n-i}^{k+1}\|_{\dH2}.
 	\end{align*}
 	By Assumption \ref{assum:R} (ii), $|sP(s)|$ and $|P(s)|<C$ for all $s\in (0,\infty)$, and thus $|s^\frac12 P(s)|\leq C$. Then, by the estimate \eqref{eqn:coro_ineq} (with $\kappa=\frac12$, $j=0$ and $k=0$), we derive
 	\begin{align*}
 		&\quad \sum_{i=0}^{n-1} {\| \alpha^{i} (\Delta T A)^{\frac{1}{2}} P( \Delta TA ) \|_{\mathcal{L}(L^2 \II )}}  \|E_{n-i}^{k+1}\|_{\dH2}\\
 		&= \|(\Delta T A)^{\frac{1}{2}}P(\Delta TA)\|_{\mathcal{L}(L^2 \II)}\|E_n^{k+1}\|_{\dH2} +  \sum_{i=1}^{n-1} {\| \alpha^{i} ( \Delta T A)^{\frac{1}{2}} P( \Delta TA ) \|_{\mathcal{L}(L^2 \II )}} \|E_{n-i}^{k+1} \|_{\dH2} \\
 		& \leq \sup_{s\in \left( 0,\infty \right)} |s^{\frac{1}{2}}P\left( s \right) | \|E_n^{k+1}\|_{\dH2} + C\sum_{i=1}^{n-1} \sup_{s\in \left( 0,\infty \right)} |\left( R\left( s \right) \right)^{i} s^{\frac{1}{2}} P(s)| \|E_{n-i}^{k+1}\|_{\dH2}\\
 		&\leq C\|E_n^{k+1} \|_{\dH2}+C\sum_{i=1}^{n-1}\frac{1}{\sqrt{i+1}}\|E_{n-i}^{k+1}\|_{\dH2}\leq C\sum_{i=0}^{n-1} \frac{1}{\sqrt{i+1}} \|E_{n-i}^{k+1}\|_{\dH2}.
 	\end{align*}
 	For the term ${\rm I}_1$, by Lemma \ref{lem:f} and the estimate \eqref{eqn:coro_ineq} (with $\kappa=\frac12$, $j=1$ and $k=0$), we obtain
 	\begin{align*}
 		\left\| {\rm I}_1 \right\|_{\dot{H}^2(\Omega)} = &\Delta T^{\frac{1}{2}} \left\| \sum_{i=1}^{n-1} \binom{i}{1} \alpha^{i-1} \beta ( \Delta T A)^{\frac{1}{2}} P(\Delta TA) A^{\frac{1}{2}}( f( U_{n-i}^{k}) - f( U_{n-i})) \right\|_{L^{2}(\Omega)} \\
 		\leq& C \Delta T^{\frac{1}{2}} \sum_{i=1}^{n-1} \binom{i}{1} \|  \alpha^{i-1} \beta( \Delta T A )^{\frac{1}{2}} P( \Delta TA)\|_{\mathcal{L}(L^2 \II )} \|E_{n-i}^{k}\|_{\dH2} \\
 		\leq & C \Delta T^{\frac{1}{2}}  \sum_{i=1}^{n-1} i^{-\frac{3}{2}} e_{n,2}^k  \leq   C \Delta T^{\frac{1}{2}} e_{n,2}^k.
 	\end{align*}
 	The bound on ${\rm I}_j$ $(j \geq 2)$ follows similarly using Theorem \ref{thm:inv_gamma_lin}. 
 \qed \end{proof}

 For the last summation in \eqref{eqn:para_whole}, we first bound the nonlinear term $\text{Res}(u,v)$. Below  $\overline{v}(s)$ denotes the exact solution to problem \eqref{eqn:pde} at time $s$ with the initial value $v \in L^2 \II$.
 \begin{lemma}\label{lem:f_-2}
 	For $q=0$ or $2$, the following estimate holds for $s \in (0,\Delta T)$,
 	\begin{align*} 
 		\| ( f( \overline{U}_{n-i}^{k} (s)) - f(U_{n-i}^{k})) - ( f( \overline{U}_{n-i} ( s)) - f( U_{n-i}))\|_{\dH{q-2}}&\leq  C \Delta T \|E_{n-i}^k\|_{\dH{q}},
 	\end{align*}
 	where the constant $C$ depends on $f$ and $(1+ \|U_{n-i}^k\|_{\dH2} + \|U_{n-i}\|_{\dH2})$, but is independent of $s$. 
 \end{lemma}
 \begin{proof}
 	We only prove the case $q = 0$, since the case $q = 2$  is similar. Let $w_\theta =  \theta {U}_{n-i}^k + (1-\theta)U_{n-i}$ for $\theta \in [0,1]$ and let $\overline{w}_\theta(s)$ be the exact solution of \eqref{eqn:pde} at time $s$ with initial data $w_\theta$. Let $D_\theta(s) =  \frac{\partial \overline{w}_\theta( s)}{\partial w_\theta}$. The generalized mean value theorem
 	\cite[Theorem 9.2.3 (ii)]{suhubi2013functional} implies
 	\begin{align*}
 		&\|( f( \overline{U}_{n-i}^{k} (s)) - f( U_{n-i}^{k}))-( f( \overline{U}_{n-i}( s ) ) - f( U_{n-i}))\|_{\dH{-2}} \\
 		\le &  \sup_{\theta\in(0,1)} \| (f'( \overline{w}_\theta ( s)) D_\theta(s) - f'( w_\theta)) E_{n-i}^{k} \|_{\dH{-2}} \\
 		\leq & \sup_{\theta\in(0,1)}  \| f'( \overline{w}_\theta ( s ) ) ( D_\theta(s) - I)  E_{n-i}^{k}\|_{\dH{-2}}  \\
 		&  +  \sup_{\theta\in(0,1)}  \left\| ( f'( \overline{w}_\theta( s)) - f'( w_\theta))( E_{n-i}^{k}) \right\|_{\dH{-2}}:= {\rm I} + {\rm II}.
 	\end{align*}
 	We first bound the term ${\rm I}$. Note that for $u \in \dH{-2}$ and $v \in H^2 \II$, 
 	\begin{align*}
\|uv\|_{\dH{-2}}&=\sup_{\psi \in \dot{H^{2}} \II} \frac{( uv,\psi)}{\| \psi \|_{\dot{H^{2}} \II}} =\sup_{\psi \in \dot{H^{2}} \II} \frac{( u,v\psi)}{\| \psi \|_{\dot{H^{2}}\II}} \\
&\leq C\sup_{\psi \in \dot{H^{2}}\II} \frac{\|u\|_{\dH{-2}  }\|v\psi \|_{\dH{2}}}{\| \psi \|_{\dot{H^{2}}\II}}.
 	\end{align*}
 	For the case $d \leq 3$, we have the well-known bilinear estimate \cite[Theorem 4.39]{adams2003sobolev},
 	\begin{align*}
 		\|v\psi \|_{\dH{2}}&\leq \| v\|_{L^{\infty}(\Omega)} \| \psi \|_{{H}^{2}\II} +
 		2{\| \nabla v\|_{L^4\II}\| \nabla \psi \|_{L^4\II}} +\| v\|_{{H}^{2} \II } \| \psi \|_{L^{\infty}\II}\\
 		& \leq C\|v\|_{H^2\II}\|\psi\|_{\dot{H}^2 \II}.
 	\end{align*}
 	Thus, there holds $\|uv\|_{\dH {-2}}
 	\leq C\|u\|_{\dH{-2}}\|v\|_{H^2 \II}$. This directly implies 
 	\begin{align*} 
 		{\rm I}&\leq {C} \|f'(\overline{w}_\theta(s))\|_{H^2 \II}\|( D_\theta(s)  -I) E_{n-i}^{k} \|_{\dH{-2}}.
 	\end{align*}
 	By Lemmas \ref{lem:f} and \ref{lem:regu_U_n}, we have
 	\begin{align*} 
 		\|f'(\overline{w}_\theta(s))\|_{H^2(\Omega)}\leq  C( 1+ \|\overline{w}_\theta(s)\|_{\dH2}) \leq C( 1+ \|U_{n-i}\|_{\dH2} + \|U_{n-i}^k\|_{\dH2}).
 	\end{align*}
 	Consequently, 
 	\begin{align}
 		{\rm I} 
 		&\leq  C( 1+ \|U_{n-i}\|_{\dH2} + \|U_{n-i}^k\|_{\dH2})\|( D_\theta  (s)-I)E_{n-i}^{k} \|_{\dH{-2}}. \label{eqn:I}
 	\end{align}
 	Note that $\overline{w}_\theta$ satisfies
 	$\overline{w}_\theta ( s ) =e^{-sA}w_\theta+\int_{0}^{s} e^{- ( s-\tau  ) A}f ( \overline{w}_\theta  ( \tau  ) )  \,\d\tau$.
 	Differentiating the identity with respect to $w_\theta$ gives
 	\begin{equation*}
 		D_\theta(s) = e^{-sA}+\int_{0}^{s} e^{-(s-\tau) A}f'( \bar{w}_\theta( \tau)) D_\theta ( \tau) \,\d\tau.
 	\end{equation*}
 	By taking the operator norm $\|\cdot\|_{\mathcal{L}(L^2\II)}$, we derive
 	\begin{equation*}
 		\| D_\theta(s)\|_{\mathcal{L}(L^2\II)} \leq 1+ C \int_{0}^{s}   \| D_\theta ( \tau)\|_{\mathcal{L}(L^2\II)} \,\d\tau.
 	\end{equation*}
 	Then Gronwall's inequality leads to the estimate
 	$\| D_\theta(s)\|_{\mathcal{L}(L^2\II)} \le {C}$ for all $s\in[0,\Delta T]$.
 	Thus we can bound the term ${\rm I}$ by 
 	\begin{align*}
 		{\rm I}&\leq  C \left\| A^{-1}( D_\theta(s) -I) E_{n-i}^{k} \right\|_{L^{2}(\Omega)} 
 		\\
 		&\leq C\| A^{-1}( e^{-sA}-I) E_{n-i}^{k}\|_{L^{2}(\Omega)} + C\int_{0}^{s} \|e^{-( s-\tau) A}{f'}(\overline{w}_\theta(\tau))D_\theta(\tau) E_{n-i}^{k} \|_{\dH{-2}}\, \d\tau \\
 		&\leq  C\| A^{-1}( e^{-sA}-I) E_{n-i}^{k} \|_{L^{2}(\Omega)} + C \int_{0}^{s} \|D_\theta( \tau)E_{n-i}^{k} \|_{L^2(\Omega)}\, \d\tau 
 		\leq C\Delta T \| E_{n-i}^{k}\|_{L^{2}(\Omega)}.
 	\end{align*}
 	Now for the term ${\rm II}$,
 	by Lemmas \ref{lem:f} and \ref{lem:regu_U_n}, we derive 
 	\begin{equation*}
 		\begin{split}
 			{\rm II} &\leq C\| \overline{w}_\theta \left( s \right) -w_\theta\|_{L^{2}(\Omega)} \| E_{n-i}^{k}\|_{L^{2}(\Omega)} 
 			\leq C s \max_{0\le \tau\le s}\|  \overline{w}_\theta' (\tau ) \|_{L^{2}(\Omega)} \| E_{n-i}^{k}\|_{L^{2}(\Omega)} \\
 			&\le C \Delta T \| w_\theta \|_{\dH2} \| E_{n-i}^{k}\|_{L^{2}(\Omega)} 
 			\le C \Delta T (\| U_{n-i}^k \|_{\dH2} +\| U_{n-i} \|_{\dH2})  \| E_{n-i}^{k}\|_{L^{2}(\Omega)}.
 		\end{split}
 	\end{equation*}
 	Combining the bounds on ${\rm I}$ and ${\rm II}$ yields the desired estimate with $q=0$.
 \qed \end{proof}

 Using Lemma \ref{lem:f_-2}, now we analyze the last summation in \eqref{eqn:para_whole}.
 \begin{lemma}\label{lem:Res}
 	For $q=0,2$, let $e_{n,q}^k=\max_{0\leq i\leq n}\|E_i^k\|_{\dH q}$. The following estimates hold
 	\begin{align*}
 		\Big\| \sum_{i=j}^{n-1} \binom{i}{j} \alpha^{i-j} \beta^{j} ~\mathrm{Res}( U_{n-i}^{k-j},U_{n-i}) \Big\|_{L^{2}(\Omega)}
 		&\leq \begin{cases}
 			C \Delta T \ln n e_{n,0}^k ,&j=0,\\ 
 			C \Delta T \gl  \ln n e_{n,0}^{k-1},&j= 1,\\
 			C \Delta T  \gl^{j}e_{n,0}^{k-j},&j\geq 2,\end{cases}\\
 		\Big\| \sum_{i=j}^{n-1} \binom{i}{j} \alpha^{i-j} \beta^{j}  \mathrm{Res}( U_{n-i}^{k-j},U_{n-i} )\Big\|_{\dH2}&\leq  C \sqrt{\Delta T} \gl^{j}  e_{n,2}^{k-j}, \quad j\geq0,
 	\end{align*}
 	where $C$ depends on $f$ and $\max_{\substack{0 \leq j \leq k \\ 0 \leq i \leq n}} ( 1 + \|U_i^j\|_{\dH2} + \| U_i \|_{\dH2})$, but is independent of $n$ and $k$.
 \end{lemma}

\begin{proof}
	First, we represent the  operator $\text{Res}(\cdot,\cdot)$ defined in \eqref{eqn:def_Res} by
	\begin{equation*}
		\begin{aligned}
			\text{Res}( u, v ) & = \left( \mathcal{F} - \mathcal{G} \right)( u ) - \left( \mathcal{F} - \mathcal{G} \right)( v ) - \beta ( u - v ) \\
			& = \int_{0}^{\Delta T} \left( e^{-(\Delta T - s)A} - P(\Delta T A) \right)  \mathrm{d} s~ ( f(u) - f(v)) \\
			&\quad  + \int_{0}^{\Delta T} e^{-(\Delta T - s)A} \left( ( f(\overline{u}(s)) - f(u) ) - ( f(\overline{v}(s)) - f(v)) \right)  \mathrm{d} s.
		\end{aligned}
	\end{equation*}
	Now we prove the $L^2(\Omega)$ estimate. By taking the $L^2(\Omega)$ norm and splitting out the case $i=j$, we arrive at
	\begin{align*}
		& \Big\| \sum_{i=j}^{n-1} \binom{i}{j} \alpha^{i-j} \beta^{j}~\text{Res}(U_{n-i}^{k-j}, U_{n-i}) \Big\|_{L^{2}(\Omega)} \leq {\rm I}_1+{\rm I}_2 +{\rm I}_3,\quad \mbox{with }\\
		{\rm I}_1=& \sum_{i=j}^{n-1} \Big\| \beta^{j} \binom{i}{j} \alpha^{i-j} \int_{0}^{\Delta T}( e^{-(\Delta T - s)A} - P(\Delta TA))  \mathrm{d} s \Big\|_{\mathcal{L}( L^2 (\Omega))}  \| f ( U_{n-i}^{k-j}  ) -  f( U_{n-i} )  \|_{L^{2}(\Omega)}, \\
		{\rm I}_2=& \Big\| \beta^{j} \sum_{i=j+1}^{n-1} \binom{i}{j} \alpha^{i-j} \int_{0}^{\Delta T} e^{-(\Delta T - s)A} \Big( \big( f ( \overline{U}_{n-i}^{k-j}  ( s  )  ) - f ( U_{n-i}^{k-j}  )  \big)  
		- \big( f (\overline{U}_{n-i} ( s  )  ) - f ( U_{n-i}  )  \big) \Big)  \mathrm{d} s \Big\|_{L^{2}(\Omega)},\\
		{\rm I}_3=& \|\beta^j\|_{\mathcal{L}(L^2 \II)}  \Big(\int_{0}^{\Delta T}\!\! \| f( \overline{U}_{n-j}^{k-j}(s)) -f( \overline{U}_{n-j}(s)) \|_{L^{2}( \Omega)} +\| f(U_{n-j}^{k-j}) -f(U_{n-j}) \|_{L^{2}( \Omega)} \d s\Big),
	\end{align*}
	where $\overline{U}_{n-i}^{k-j}(s)$ and $\overline{U}_{n-i}(s)$ denote the exact propagators at time $s$ with the initial values $U_{n-i}^{k-j}$ and $U_{n-i}$, respectively. The term ${\rm I}_1$ can be bounded as 
	\begin{align*}
		{\rm I}_1  &\leq Ce_{n,0}^{k-j}{\rm I}_1', \quad \mbox{with }{\rm I}_1'= \sum_{i=j}^{n-1} \Big\| \beta^j \int_{0}^{\Delta T}\binom{i}{j}\alpha^{i-j}  \left( e^{ -(\Delta T-s)A} - P\left( \Delta T A \right) \right) \mathrm{d} s \Big\|_{\mathcal{L}(L^2 (\Omega))}.
	\end{align*}
	Under Assumption \ref{assum:R}, we have  
	\begin{align*}
		{\rm I}_1'
		&\leq \Delta T \sum_{i=j}^{n-1} \Big\| \binom{i}{j} \alpha^{i-j} \beta^{j}  \big(  ( \Delta TA  )^{-1}  ( I-e^{-\Delta T A} ) -P ( \Delta TA  )  \big) \Big\|_{\mathcal{L} (L^2(\Omega))} \\
		&\leq \Delta T \sum_{i=j}^{n-1} \binom{i}{j} \sup_{s\in \left( 0,\infty \right)} \left| \left( R\left( s \right) \right)^{i-j} \left( e^{-s}-R\left( s \right) \right)^{j} \left( \frac{1-e^{-s}}{s} -P\left( s \right) \right) \right| \\
		&\leq \begin{cases}
			C \Delta T   \ln n, & j = 0 \quad \text{(\eqref{eqn:coro_ineq} with $\kappa=0$, $j=0$ and $k=1$)}, \\
			C \Delta T  \ln n \gl, & j = 1 \quad \text{(\eqref{eqn:coro_ineq_2} with $\kappa=0$ and $k=1$)}, \\
			C \Delta T  \gl^{j}, & j \ge 2 \quad \text{(Theorem \ref{thm:inv_gamma_lin})}.
		\end{cases}
	\end{align*}
	The term ${\rm I}_2$ can be bounded by
	\begin{align*}
		{\rm I}_2
		&\leq  \max_{0\leq i \leq n} \sup_{s \in \left( 0, \Delta T \right)} \big\| \big ( f( \overline{U}_{n-i}^{k-j} ( s )) - f( U_{n-i}^{k-j} ) \big) -  \big ( f ( \overline{U}_{n-i}  ( s  )  ) - f ( U_{n-i} )  \big ) \big \|_{\dH{-2}} \\
		&\qquad \times \sum_{i=j+1}^{n-1} \frac{1}{\Delta T} \int_{0}^{\Delta T} \binom{i}{j} \left\| \beta^j \alpha^{i-j} e^{ -(\Delta T-s)A} A \Delta T \right\|_{\mathcal{L}(L^2 (\Omega))} \mathrm{d} s=: {\rm I}_2 '\cdot {\rm I}_2''.
	\end{align*}
	By  Lemma \ref{lem:f_-2}, there exists $C$ depending on $f$ and $\max_{\substack{0 \leq j \leq k \\ 0 \leq i \leq n-1}} ( 1 + \|U_i^j\|_{\dH2} + \| U_i \|_{\dH2})$ such that
	${\rm I}_2'
	\leq C\Delta T e_{n-1,0}^{k-j}$.
	Further, under Assumption \ref{assum:R}, we obtain
	\begin{align*}
		{\rm I}_2''
		&= \sum_{i=j+1}^{n-1} \int_{0}^{1} \binom{i}{j} \| \beta^{j} \alpha^{i-j} e^{-\left( 1-t \right) \Delta T A} \Delta T A\|_{\mathcal{L} (L^2 \II )} \, \d t \\
		&\leq \sum_{i=j+1}^{n-1} \int_{0}^{1} \binom{i}{j} \sup_{s\in \left( 0,\infty \right)} |\left( e^{-s}-R\left( s \right) \right)^{j} {\left( R\left( s \right) \right)^{i-j} e^{-\left( 1-t \right) s}s| \, \d t} \\
		&\leq \begin{cases}
			C \ln n,  & j = 0 \quad \text{(\eqref{eqn:coro_ineq} with $\kappa=1$, $j=0$ and $k=0$)}, \\ 
			C \ln n  \gl , & j = 1 \quad \text{(\eqref{eqn:coro_ineq_2}  with $\kappa=1$ and $k=0$)}, \\
			C   \gl^{j}, & j \ge 2 \quad \text{(Theorem \ref{thm:inv_gamma_lin})}.
		\end{cases}
	\end{align*}
	The third term ${\rm I}_3$ can be bounded directly by
	\begin{align*}
		{\rm I}_3
		&\leq \gl^j C \Delta T \Big(\sup_{s\in ( 0,\Delta T)} \| \overline{U}_{n-j}^{k-j}( s) -\overline{U}_{n-j}(s) \|_{L^2(\Omega )} +\| E_{n-j}^{k-j}\|_{L^2(\Omega)}\Big)
		\leq C \Delta T \gl^j\| E_{n-j}^{k-j}\|_{L^2 \II}.
	\end{align*}
	Combining the preceding estimates yields the desired $L^2(\Omega)$ bound.
	Now we derive the estimate in the $\dot{H}^2 \II$ norm. By repeating the preceding argument with Lemma \ref{lem:f}, we obtain
	\begin{align*} 
		\Big\| \sum_{i=j}^{n-1}& \binom{i}{j} \alpha^{i-j} \beta^{j} A ~\text{Res}( U_{n-i}^{k-j}, U_{n-i}) \Big\|_{L^{2}(\Omega)} 
		\leq C e_{n,2}^{k-j}\cdot \mathrm{II}_1  + \mathrm{II}_2 \cdot \mathrm{II}_3 + C \sqrt{\Delta T} \gl^j \|E_{n-j}^{k-j}\|_{\dH2}, \\
		\mbox{with } &
		{\rm II}_1 = e_{n,2}^{k-j}\cdot \sum_{i=j}^{n-1} \Big\| \beta^{j} \binom{i}{j} \alpha^{i-j} \int_{0}^{\Delta T} A^{\frac{1}{2}}( e^{-(\Delta T-s)A}-P(\Delta TA) ) \, \d s \Big\|_{\mathcal{L} (L^2 \II)},\\
		&{\rm II}_2  = \max_{0\leq i \leq n} \sup_{s \in ( 0, \Delta T )}\|( f( \overline{U}_{n-i}^{k-j} ( s)) - f( U_{n-i}^{k-j} ) ) - ( f( \overline{U}_{n-i}( s)) - f( U_{n-i} ))\|_{L^2(\Omega)} , \\
		&{\rm II}_3 =\sum_{i=j+1}^{n-1} \frac{1}{\Delta T} \int_{0}^{\Delta T} \binom{i}{j} \| \beta^j \alpha^{i-j} e^{-( \Delta T-s )A} A \Delta T \|_{\mathcal{L}(L^2(\Omega))}\,\d s.
	\end{align*}
	For the term ${\rm II}_1$, we have 
	\begin{align*} 
		{\rm II}_1  & \le C \sqrt{\Delta T}\sum_{i=j}^{n-1}  \Big\| \binom{i}{j} \alpha^{i-j} \beta^{j} \left( \Delta TA \right)^{\frac{1}{2}}  ((  \Delta TA)^{-1}( I - e^{-\Delta T A}  ) - P ( \Delta TA  ) \big) \Big\|_{\mathcal{L}(L^2(\Omega))} \\
		&\leq C \sqrt{\Delta T} \sum_{i=j}^{n-1} \binom{i}{j}\sup_{s\in (0,\infty)}  \Big| ( R( s) )^{i-j} ( e^{-s} - R( s) )^{j} s^{\frac{1}{2}} \Big( \frac{1 - e^{-s}}{s} - P( s ) \Big)  \Big|\\
		&\leq  \begin{cases}
			C \sqrt{\Delta T},  & j = 0 \quad \text{(\eqref{eqn:coro_ineq} with $\kappa=\frac12$, $j=0$ and $k=1$)}, \\ 
			C \sqrt{\Delta T} \gl , & j = 1 \quad \text{(\eqref{eqn:coro_ineq_2} with $\kappa=\frac12$ and $k=1$)}, \\
			C  \sqrt{\Delta T} \gl^{j}, & j \ge 2 \quad \text{(Theorem \ref{thm:inv_gamma_lin})}.
		\end{cases}
	\end{align*}
	Meanwhile, Lemma \ref{lem:f_-2} implies 
	${\rm II}_2   \leq C \Delta T e_{n,2}^{k-j},$ with $C$ depending on $f$ and $\max_{\substack{0 \leq j \leq k \\ 0 \leq i \leq n-1}} ( 1 + \|U_i^j\|_{\dH2} + \| U_i \|_{\dH2})$. Moreover,
	repeating the argument for the estimate in the $L^2(\Omega)$ norm yields 
	\begin{align*}
		{\rm II}_3   &\leq  \begin{cases}
			C \ln n, & j = 0 \quad \text{(\eqref{eqn:coro_ineq} with $\kappa =1, j=0$ and $k=0$)}, \\
			C \gl \ln n, & j = 1 \quad \text{(\eqref{eqn:coro_ineq_2} with $\kappa=1$ and $k=0$)}, \\
			C \gl^{j}, & j \geq 2 \quad \text{(Theorem~\ref{thm:inv_gamma_lin})}.
		\end{cases}
	\end{align*}
	Combining the preceding estimates yields the estimate  for $j \geq 1$. For $j=0$, since $\ln n ~\Delta T \leq C \Delta T^\frac{1}{2}$, the result also follows.
\qed \end{proof}

\begin{remark}
	With the exponential Euler method as the CP, Brehier and Wang \cite{BrehierWang:2020} derived a result analogous to the case $j=0$ in Lemma \ref{lem:Res} under the $L^2 \II$ norm. Since the linear component of problem \eqref{eqn:pde} is exactly propagated within exponential integrators, the analysis is much simpler.
\end{remark}

Combining the relation \eqref{eqn:para_whole} with Lemmas \ref{lem:lin}, \ref{lem:f-f} and \ref{lem:Res} gives the main result in $\dot{H}^2 \II$.
\begin{theorem}\label{thm:main_thm_H2}
	Suppose that $U_n^0\in \dot{H}^2(\Omega)$. Then the error $E_n^k=U_n^k -U_n$ satisfies
	\begin{equation}\label{eqn:conv_H2}
		\max_{0 \leq n\leq N_c-1}\|E_{n+1}^{k+1}\|_{\dH2}\leq C\big( \gl +C_1 \sqrt{\Delta T}  \big)^{k+1},
	\end{equation}
	where $C$ and $C_1$ are independent of $n$, $k$, $N_c$ and $\Delta T$.
\end{theorem}

\begin{proof}
	We proceed by mathematical induction on the iteration index $k$. Suppose that $$\max_{0\leq n \leq N_c-1}\|E_{n+1}^{k_0+1}\|_{\dH2}\leq C( \gl +C_1\sqrt{ \Delta  T})^{k_0+1}$$ holds for all $k_0 \leq k-1$. This and Lemma \ref{lem:regu_U_n} indicate that with small $\Delta T$, there holds
	\begin{equation*}
		\max_{0\leq n \leq N_c-1}\|U_{n+1}^{k_{0}+1}\|_{\dot{H}^{2}(\Omega)} 
		\leq \max_{0 \leq n \leq N_c-1} \| U_{n+1}\|_{\dot{H}^{2}(\Omega)} 
		+\max_{0\leq n \leq N_c-1}\| E_{n+1}^{k_{0}+1}\|_{\dH2}\leq C,
	\end{equation*}
	for all $k_0 \leq k-1$, and the bound is uniform in $k_0$, $k$ and $n$. 
	Then for $k+1$, the identity
	\eqref{eqn:para_whole} gives 
	\begin{align*}
		\max_{0\leq n \leq N_c-1}\|E_{n+1}^{k+1}\|_{\dH2} {\le} & \sum_{j=0}^{k} \sum_{i=j}^{N_c-2}\left\| \binom{i}{j} \alpha^{i-j} \beta^{j} A\Delta TP( \Delta TA) ( f( U_{N_c-1-i}^{k+1-j}) - f( U_{N_c-1-i}) )\right\|_{L^2 \II} \\
		& +\max_{0\leq n\leq N_c-1} \left\| \sum_{i=k}^{n} \binom{i}{k} \alpha^{i-k} \beta^{k+1} AE_{n-i}^{0} \right\|_{L^2(\Omega)} \\
		&+\sum_{j=0}^{k} \sum_{i=j}^{N_c-2} \left\|\binom{i}{j} \alpha^{i-j} \beta^{j}A~ \mathrm{Res}( U_{N_c-1-i}^{k-j}, U_{N_c-1-i} )\right\|_{L^2(\Omega)}.
	\end{align*}
	By combining Lemmas  \ref{lem:lin}, \ref{lem:f-f} and \ref{lem:Res}, we  obtain 
	\begin{align*}
		&\max_{0\leq n \leq N_c-1}\|E_{n+1}^{k+1}\|_{\dH2}\leq \Big( C \Delta T^{\frac{1}{2}} \sum_{i=0}^{N_c-2} \frac{1}{\sqrt{i+1}} \|E_{N_c-1-i}^{k+1}\|_{\dH2}+ C\Delta T^{\frac12} \max_{0\leq i \leq N_c} \|E_{i}^k\|_{\dH2} \\ &\quad+\sum_{j=2}^{k} C \Delta T^{\frac{1}{2}}\gl^{j} \max_{0\leq i \leq N_c} \|E_{i}^{k+1-j}\|_{\dH2} \Big)+\gl^{k+1} \sum_{i=0}^{N_c} \|E_{i}^{0}\|_{\dH2}+\sum_{j=0}^{k} C \Delta T^{\frac{1}{2}}\gl^{j}  \max_{0\leq i \leq N_c} \|E_{i}^{k-j}\|_{\dH2}\\
		&\leq C\Delta T\sum_{i=0}^{N_c-2}T_{i+1}^{-\frac12} \max_{0\leq n \leq N_c-1-i} \|E_{n}^{k+1}\|_{\dH2} +\gl^{k+1}\sum_{i=0}^{N_c}\|E_i^0\|_{\dH2} +C\Delta T^{\frac{1}{2}}\sum_{j=0}^k \gl^j \max_{0\leq i\leq N_c} \|E_i^{k-j}\|_{\dH2}.
	\end{align*}
	By applying Lemma \ref{lem:x_k_1} with $a=\gl^{k+1}\sum_{i=0}^{N_c}\|E_i^0\|_{\dH2} + C\sqrt{\Delta T}\sum_{j=0}^k \gl^j \max_{0\leq i\leq N_c} \|E_i^{k-j}\|_{\dH2}$ and $x_{n}=\max_{0\leq i\leq n}\|E_{i}^{k+1}\|_{\dH2}$, we obtain
	\begin{equation*}
		\max_{0\leq n\leq N_c -1}\|E_{n+1}^{k+1}\|_{\dH2}\leq C\sqrt{T}{e^{ C{T}}}\Big( \gl^{k+1} \sum_{i=0}^{N_c} \|E_{i}^{0}\|_{\dH2}+C \sqrt{\Delta T}\sum_{j=0}^{k} \gl^{j}  \max_{0\leq i \leq N_c} \|E_{i}^{k-j}\|_{\dH2} \Big).
	\end{equation*}
	This recursion and Lemma \ref{lem:x_k} yield the desired bound, with $C_1 = C\sqrt{T}e^{CT}$.
\qed \end{proof}

Theorem \ref{thm:main_thm_H2} directly implies the uniform boundedness of $\|U_{n}^{k}\|_{\dot{H}^2\II }$ for all $n$ and $k$.
\begin{corollary}\label{coro:regu_H2}
	Suppose that $ U_n^0 \in \dH2$ for all $n$. Then the parareal solution $U_n^k$ by Algorithm \ref{alg:para} satisfies
	$\|U_{n}^{k}\|_{\dot{H}^2\II } \leq C$, with the constant $C$ independent of $n$ and $k$.
\end{corollary}

Last, we bound the parareal error in $L^2 \II$, which exhibit a better rate in $\Delta t$ than that in Theorem \ref{thm:main_thm_H2}. 

\begin{theorem}\label{thm:main_thm_L2}
	Suppose that $U_n^0\in \dot{H}^2(\Omega)$ for all $n$. Then the parareal error $E_n^k$ satisfies
	\begin{equation*} \max_{0\leq n \leq N_c-1}\|E_{n+1}^{k+1}\|_{L^2(\Omega)}\leq C (\gl +   C_2  \Delta T \ln N_c )^{k+1},
	\end{equation*}
	where $C$ and $C_2$ are independent of $n$, $k$, $N_c$ and $\Delta T$.
\end{theorem}
\begin{proof} By Corollary~\ref{coro:regu_H2}, $\|U_n^k\|_{\dH2} \leq C$ for all $n$ and $k$. By taking the $L^2(\Omega)$ norm on both sides of the identity \eqref{eqn:para_whole}, and by Lemmas \ref{lem:f-f},  \ref{lem:lin} and \ref{lem:Res}, we obtain
	\begin{align*}
		\max_{0\leq n\leq N_c -1}\| E_{n+1}^{k+1}\|_{L^{2}(\Omega)}
		&\leq C\Delta T\sum_{i=1}^{N_c} \| E_{i}^{k+1}\|_{L^{2}(\Omega)} +\gl^{k+1} \sum_{i=0}^{N_c} \| E_{i}^{0}\|_{L^{2}(\Omega)} \\
		&\quad +C\Delta T \ln N_c \max_{0 \leq i \leq N_c } \| E_{i}^{k}\|_{L^{2}(\Omega)} +C\Delta T \ln N_c \sum_{j=1}^{k} \gl^j\max_{0 \leq i\leq N_c} \| E_{i}^{k-j}\|_{L^{2}(\Omega)}.
	\end{align*}
	The discrete Gronwall's inequality on the $(k+1)$-th iteration gives
	\begin{align*}
		\max_{0\leq n \leq N_c -1}\|E_{n+1}^{k+1}\|_{L^{2}\II} 
		&\leq e^{C T} \Big( \gl^{k+1} \sum_{i=0}^{N_c} \| E_{i}^{0}\|_{L^{2}\II}  + C\Delta T\ln N_c\max_{0 \leq i \leq N_c} \| E_{i}^{k}\|_{L^{2}\II}\\
		&\quad+ C \Delta T \ln N_c\sum_{j=1}^{k} \gl^j\max_{0\leq i \leq N_c} \| E_{i}^{k-j}\|_{L^{2}\II} \Big).
	\end{align*}
	Letting $x_{k}=\max_{0\leq i\leq N_c} \| E_{i}^{k}\|_{L^{2}\II}$ lead to 
	\begin{equation*}
		x_{k+1}\leq \gl^{k+1}f_{0}+\sum_{j=0}^{k} g_{j}x_{k-j}, \quad \text{with}~
		f_{0}=e^{CT} {\sum_{i=0}^{N_c} \| E_{i}^{0}\|_{L^{2}\II}}~\text{and}~ g_{j}= e^{CT}\gl^{j} C\Delta T \ln N_c.
	\end{equation*}
	Then Lemma \ref{lem:x_k} implies 
	the desired assertion, with $C_2 = Ce^{CT}  $.
\qed \end{proof}

\begin{remark}
	If problem \eqref{eqn:pde} is linear with a source $f(t)$, then by Lemma \ref{lem:expression_para}, the recursion for the parareal error reduces to $$E_{n+1}^{k+1}=\beta^{k+1} \sum_{i=k}^n \binom{i}{k} \alpha^{i-k}E_{n-i}^0.$$ 
	Then Lemma \ref{lem:lin} implies 
	$$\|E_{N_c}^{k+1}\|_{L^2\II } \leq C \gl^{k+1},$$
	which aligns with the linear convergence for linear parabolic equations.
\end{remark}

 \section{Numerical experiment}\label{sec:numerical}
 Now we numerically illustrate the theoretical finding in Theorem \ref{thm:main_thm_L2}. We 
 fix the domain $\Omega=(-1,1)$ and the final time $T=2$. Consider the semilinear parabolic equation
 \begin{equation*}
 	\partial_{t} u(x,t)=\partial_{xx} u(x,t)+c_L u(x,t)\left( 1-u(x,t)^2 \right) + g(x,t),
 \end{equation*}
 equipped with the forcing term $g(x,t)=\cos (\frac{\pi}{2} x) \cos (t)$ and a zero Dirichlet boundary condition. 
 We take the initial data 
 $$u_0 = \sqrt{x+1} \sin(2\pi x) \in \dot{H}^{2-\epsilon} \II\quad \text{for any small}~~\epsilon >0$$   
 Note that the constant $c_L$ in the nonlinear term controls the strength of nonlinearity. 
 
 We discretize the problem using the central finite difference method for the diffusion term $\partial_{xx}u$, with a mesh size of $h=1/256$, and employ SDIRK3 as the FPs. The temporal discretization parameters are set with a fine step size of $\Delta t = 1.5625\times 10^{-4}$. The parareal algorithm is initialized using the initial data $U_0^0$. We present the $L^2(\Omega)$ errors and compare different CPs across the settings $c_L\in \{1,5,10\}$ and $\Delta T \in \{ 0.05,0.0125,0.003125\}$ in Fig.~\ref{fig:Ex1}. 
 
 First, we observe that the $L^2$ errors exhibit steady linear convergence across all the settings. The performance of the parareal algorithm deteriorates as the degree of nonlinearity $c_L$ increases. Nonetheless, reducing the coarse time step size $\Delta T$ allows the convergence factor for each method to approach $\gl$, i.e., the convergence factor in the linear case; see Remark \ref{rem:gamma_lin} for the concrete values for the three CPs. These empirical observations fully align with the theoretical estimate in Theorem \ref{thm:main_thm_L2}.
 
 More precisely, for the BE method, decreasing the coarse time step size  $\Delta T$ improves the performance of the parareal algorithm. However, when $\Delta T$ becomes sufficiently small, the parareal algorithm tends to the convergence factor  $\gl \approx 0.298$, which is relatively large compared to the other two CPs and severely limits the convergence speed. For the two-stage Lobatto IIIC method, the CP resolves the nonlinear terms effectively and appears to be more robust with respect to the nonlinearity parameter $c_L$. When $\Delta T$ is sufficiently small, it still recovers the linear convergence factor, which ultimately limits further improvements in convergence speed.
 In sharp contrast, the OCP significantly reduces the linear convergence factor $\gl$ compared to both the BE method and the two-stage Lobatto IIIC method. These observations highlight the importance of designing appropriate CPs and conducting convergence analysis for linear problems to enhance the performance of the parareal algorithm.
 
 \begin{figure}[htbp!]
 	\centering
 	\includegraphics[width=0.99\linewidth]{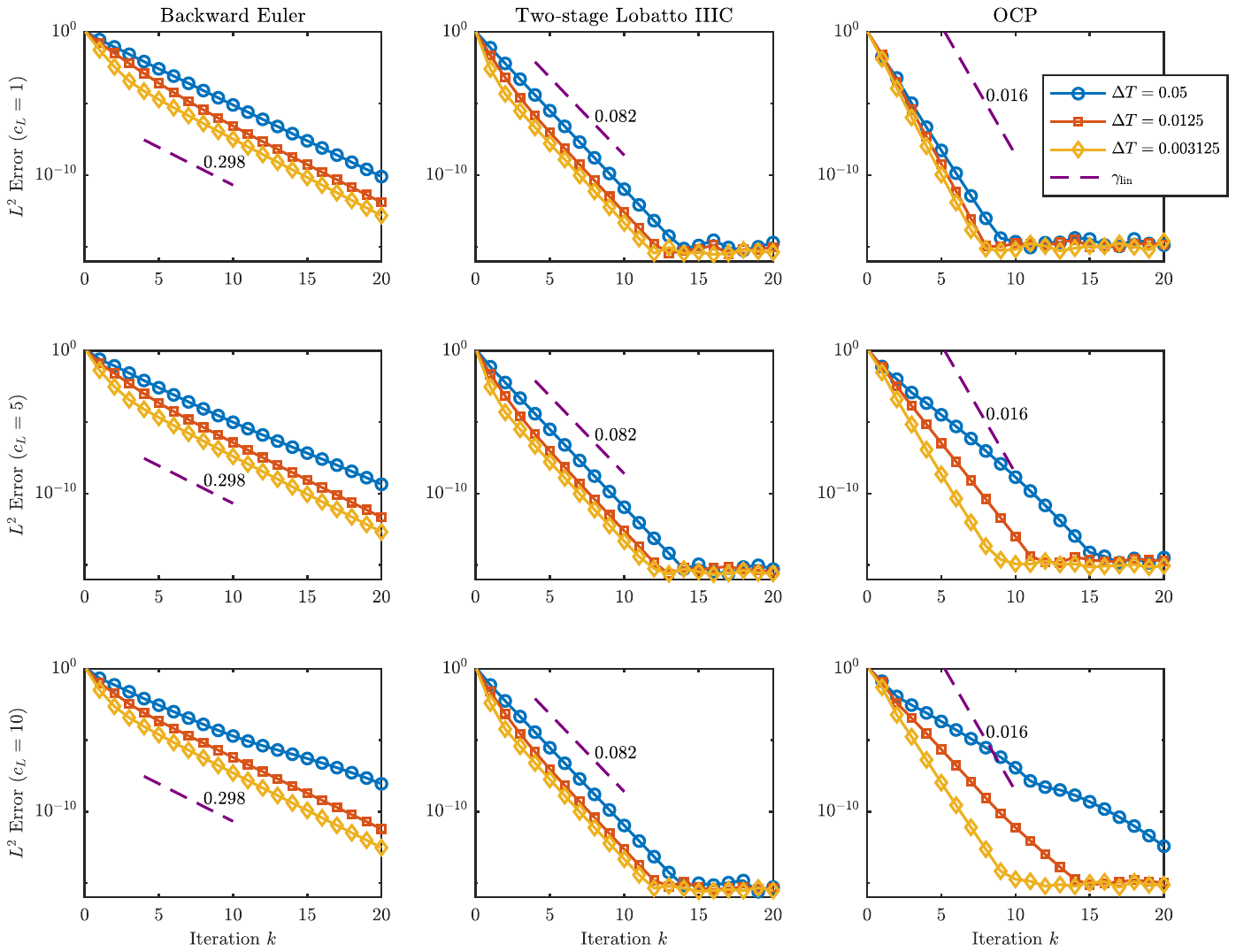}\vspace{-30pt}
 	\caption{The maximum $L^2$ error for three CPs versus the iteration $k$ over the setting $c_L\in \{1,5,10\}$ and $\Delta T \in \{ 0.05,0.0125,0.003125\}$. Left: BE; Middle: Two-stage Lobatto IIIC method; Right: OCP.} 
 	\label{fig:Ex1}
 \end{figure}
 
 \section{Conclusion}
 In this work, we have provided a novel convergence analysis of the parareal algorithm for solving semilinear parabolic equations with the $H^2$ initial data. 
 The algorithm employs a single-step coarse propagator and first-order linearization of the nonlinear term. 
 We have established a linear convergence rate of the parareal algorithm, and derived the convergence factor of the form $\gl + C_f\Delta T |\ln \Delta T|$, with the constant $C_f$ depending on the nonlinear function $f$. The theoretical findings have been also confirmed by numerical experiments. This result provides the much-needed theoretical justifications on the empirical observation in prior studies, and represents an important contribution to the theoretical analysis of the parareal algorithm.
 The analysis framework holds also the potential for analyzing more challenging scenarios, e.g., when the initial data is weaker (e.g., $u_0 \in L^2(\Omega)$) or when the nonlinearity is stronger. 
 \appendix
 
 \section{Proof of Lemma \ref{lem:f}}\label{sec:proof2}
 Since $f$ satisfies Assumption \ref{assum:f}, we have $\| f'(u)\|_{L^\infty \II} \leq C$ and $\| f'(u) - f'(v) \|_{L^2 \II} \leq C \|u-v\|_{L^2 \II}$. 
 For the estimate \eqref{eqn:basic-est1},
 \begin{align*}
 	\|( f'(u) - f'(v))( h) \|_{\dot{H}^{-2}\II}  &\leq C \sup_{\psi \in \dot{H}^{2}\II} \frac{(( f'(u) - f'(v))(h), \psi )}{\| \psi \|_{\dot{H}^{2 }\II}} \\
 	&\leq C\|( f'(u) - f'(v))(h) \|_{L^{1}\II} \sup_{\psi \in \dot{H}^2 \II}\frac{\| \psi \|_{L^{\infty}\II}}{\| \psi \|_{\dot{H}^{2}\II}}.
 \end{align*}
 By the Sobolev embedding $\| \psi \|_{L^{\infty}\II} \leq C\| \psi \|_{\dot{H}^{2}\II}$ (for $d=1,2,3$), we have
 \begin{align*}
 	&\|( f'(u) - f'(v))(h) \|_{\dot{H}^{-2}\II}  \leq C \|( f'(u) - f'(v))(h) \|_{L^{1} \II} \\
 	\leq& C \| f'(u) - f'(v) \|_{L^{2}\II} \| h \|_{L^{2} \II } 
 	\leq  C \| u - v \|_{L^{2}\II} \| h \|_{L^{2} \II}.
 \end{align*}
 This prove the estimate \eqref{eqn:basic-est1}.
 Next, the estimate \eqref{eqn:basic-est2} follows by
 \begin{equation*}
 	\|( f'(u) -f'(v))(h) \|_{L^{2} \II} \leq \| f'(u) -f'(v) \|_{L^{2} \II} \| h\|_{L^{\infty}\II} \leq C\| u-v\|_{L^{2} \II} \|h\|_{{H}^2 \II}.
 \end{equation*}
 For the estimate \eqref{eqn:basic-est3}, we obtain for $u,v \in \dot{H}^{2}(\Omega)$,
 \begin{align*}
 	\| \nabla(f(u) - f(v)) \|_{L^2(\Omega)} 
 	&\leq \| (f'(u) - f'(v)) \nabla u \|_{L^{2}(\Omega)} + \| f'(v) \|_{L^{\infty}(\Omega)} \| u - v \|_{\dot{H}^{1}(\Omega)} \\
 	&\leq \| f'(u) - f'(v) \|_{L^{\infty}(\Omega)} \| u \|_{\dot{H}^{1}(\Omega)} + C \| u - v \|_{\dot{H}^{1}(\Omega)} \\
 	&\leq C  \| u - v \|_{L^{\infty}(\Omega)} + C\| u - v \|_{\dot{H}^{1}(\Omega)} + C \| u - v \|_{\dot{H}^{2}(\Omega)}
 	\leq C \|u-v\|_{\dH2}.
 \end{align*}
 Further, since $f(u) - f(v)$ satisfies the zero boundary condition, $f(u) - f(v) \in \dH1$. Finally, for $u \in \dH2$, by the Gagliardo-Nirenberg interpolation inequality  $\| \nabla u\|_{L^{4}\II} \leq C\| u\|_{L^{\infty}\II }^{1/2} \| \Delta u\|_{L^{2} \II }^{1/2}$ \cite{nirenberg1966extended} and Assumption \ref{assum:f}, we have 
 \begin{align*}
 	\| f'(u) \|_{H^2(\Omega)} 
 	&\leq C(\| f'(u) \|_{L^2(\Omega)} + \| \nabla f'(u) \|_{L^2(\Omega)} + \| \Delta  f'(u) \|_{L^2(\Omega)}) \\
 	&\leq C\|f'(u)\|_{L^\infty \II} + C\| f''(u) \|_{L^\infty(\Omega)} \| \nabla u \|_{L^2(\Omega)} \\
 	&\quad + C\| f'''(u) \|_{L^\infty(\Omega)} \| \nabla u \|_{L^4(\Omega)}^2 
 	+ C\| f''(u) \|_{L^\infty(\Omega)} \|   u \|_{H^2(\Omega)} 
 	\leq C (1 + \| u \|_{\dot{H}^2(\Omega)} ).
 \end{align*}
 $\qed$

 \section{Proof of Theorem \ref{thm:inv_gamma_lin}}\label{sec:proof1}  The proof is lengthy and very technical, and it is divided into three steps. Let 
 \begin{equation}\label{eqn:eta}
 	\eta = \sup_{s \in (0, \infty)} \frac{|e^{-s} - R(s)|}{( 1 - |R(s)|)^2}\quad \mbox{and}\quad r_{0}=\inf \Big\{ r:\,\Big( 1-\frac{1}{r+1} \Big)^{r} \frac{4\eta}{r+1} \leq \gl,\,r\geq 2 \Big\}.
 \end{equation}
 By Assumption \ref{assum:R} (i), $\eta$ is finite. The parameter $r_0$ is chosen so that the summation of index over $r_0 j$ is bounded by $\frac{C}{\sqrt{j}}\gl^j$; see the end of Step (i).
 
 \medskip
 \noindent\textbf{Step (i):} Consider $i$ from $r_0j+1$ to $\infty$. By Stirling's formula $\sqrt{2\pi n} \left( \frac{n}{e} \right)^{n} \leq n!\leq \sqrt{2\pi n} \left( \frac{n}{e} \right)^{n} e^{\frac{1}{12n}}$ \cite[(1.53)]{spencer2014asymptopia},  we have 
 \begin{equation}\label{eqn:stirling}
 	\binom{i}{j} =\frac{i!}{j!\left( i-j \right) !} \leq \frac{\sqrt{2\pi i} \left( \frac{i}{e} \right)^{i} e^{\frac{1}{12}}}{\sqrt{2\pi j} \left( \frac{j}{e} \right)^{j} \sqrt{2\pi \left( i-j \right)} \left( \frac{i-j}{e} \right)^{i-j}} \leq \frac{e^{\frac{1}{12}}}{\sqrt{2\pi}} \sqrt{\frac{i}{j\left( i-j \right)}} \frac{i^{i}}{j^{j}\left( i-j \right)^{i-j}}.
 \end{equation}
 Then by the elementary inequality 
 \begin{equation}\label{eqn:basic-ineq0}
 	\theta^{a} \left( 1-\theta \right)^{b} \leq \frac{a^{a}b^{b}}{\left( a+b \right)^{a+b}},\quad \mbox{for }\theta\in [0,1], a,b>0,
 \end{equation}
 and the definition of $\eta$ in \eqref{eqn:eta}, we deduce 
 \begin{align*}
 	{\rm I}:=&\sum_{i=r_{0}j+1}^{\infty} \binom{i}{j} \sup_{s\in \left( 0,\infty \right)} |\left( R\left( s \right) \right)^{i-j} \left( e^{-s}-R\left( s \right) \right)^{j} | 
 	\leq  \eta^{j} \sum_{i=r_{0}j+1}^{\infty} \binom{i}{j} \sup_{s\in \left( 0,\infty \right)} |R\left( s \right) |^{i-j}\left( 1-|R\left( s \right) | \right)^{2j} \\
 	\leq &\, C\eta^{j} \sum_{i=r_{0}j+1}^{\infty} \frac{1}{\sqrt{j}} \frac{i^{i}}{j^{j}\left( i-j \right)^{i-j}} \frac{\left( i-j \right)^{i-j} \left( 2j \right)^{2j}}{\left( i+j \right)^{i+j}}
 	= C\eta^{j} \sum_{i=r_{0}j+1}^{\infty} \frac{1}{\sqrt{j}} \frac{i^{i}}{(i+j)^i} \frac{\left( 4j \right)^{j}}{\left( i+j \right)^{j}}.
 \end{align*}
 Now for $i\geq r_0j+1$, we have
 \begin{equation*}
 	\frac{i^i}{(i+j)^i} = \left(1-\frac{1}{1+\frac{i}{j}}\right)^{\frac{i}{j}\cdot j}\leq \left(1-\frac{1}{1+r_0}\right)^{r_0j},
 \end{equation*}
 since the function $r\mapsto (1-\frac{1}{1+r})^r$ is decreasing.
 Consequently, by the definition of $r_0$ in \eqref{eqn:eta}
 \begin{align*}
 	{\rm I}\leq & 
 	\frac{C}{\sqrt{j}} \left( 4\eta\left( 1-\frac{1}{r_{0}+1} \right)^{r_{0}}  \right)^{j} \sum_{i=r_{0}j+1}^{\infty} \frac{j^{j}}{\left( i+j \right)^{j}} \\
 	\leq& \, C\left( 4\eta\left( 1-\frac{1}{r_{0}+1} \right)^{r_{0}} \right)^{j} \frac{j^{j}}{\sqrt{j} \left( j-1 \right) \left( j\left( r_{0}+1 \right) \right)^{j-1}}  \leq \frac{C}{\sqrt{j}} \gl^{j}.
 \end{align*}
 \textbf{Step (ii):} 
 Consider the index $i$ ranging from $j$ to $r_0j$.
 Let $\{r_k^*\}_{k=1}^K$ be the zeros of $h(r)$ (cf. \eqref{hr}). By Assumption \ref{assum:R} (iii) and Remark \ref{rem:ass}, there exists a $\delta>0$ such that for all $r \in \bigcup_{k=1}^K (r_k^* - \delta, r_k^* + \delta)$, the inequality \eqref{eqn:f_lower bound} holds, and moreover $r_1^\ast - 1>\delta$. 
 Let $I_K = [1, r_0] \setminus \bigcup_{k=1}^K (r_k^* - \delta, r_k^* + \delta)$.
 Therefore, if $r \in I_K$, then we have $h(r)\geq h_0>0$. Now consider the following splitting:
 \begin{align}\label{eqn:Irj}
 	{\rm II=}&\sum_{\substack{i \in I_{K}j \cap \mathbb{N}}} \binom{i}{j}
 	\sup_{s \in (0,\infty)} \left|
 	R(s)^{i-j} \left( e^{-s} - R(s) \right)^j
 	\right| \nonumber \\
 	\leq &\sup_{s\in(0,\infty)}|e^{-s}- R(s)|^j +
 	\sum_{\substack{i \in I_{K}j \cap \mathbb{N} \\ i \neq j}} 
 	\binom{i}{j}
 	\sup_{s\in(0,\infty)} \left|
 	R(s)^{i-j} \left( e^{-s} - R(s) \right)^j
 	\right|.
 \end{align}
 Note that
 \begin{equation*}
 	\sup_{s\in(0,\infty)}|e^{-s}- R(s)|^j  \le \sup_{s\in(0,\infty)} \left|\frac{e^{-s}- R(s)}{1-|R(s)|}
 	\right|^j = \gl^j.
 \end{equation*}
 Let $s_{i,j}$ maximize $|Q(s,i/j)|$, with $Q(s,r)=R(s)^{r-1} (e^{-s}-R(s))$. We claim that for $i \in I_K j \cap \mathbb{N}$, there exists $c_0$ independent of $j$ such that 
 \begin{align} \label{eqn:c0-gamma}
 	|\gamma(s_{i,j})|\leq c_0 < \gl.
 \end{align}
 Suppose that no such $c_0$ exists. Then, for $j\geq 2$, let $i\left( j \right) := \arg\max_{i \in I_{K}j \cap \mathbb{N}} \left|\gamma \left( s_{i,j} \right) \right|$ and set $s_j := s_{i(j),j}$. Recall that $s_0$ is the unique maximizer of $|\gamma(s)|$, cf. Assumption \ref{assum:R}(iv). Now consider two cases: (i) There exists some $j_0$ such that $s_{j_0}=s_0$. Then $s_{k,j_0}=s_0$ for all $k \in \mathbb{N}$; (ii) $s_{j}\neq s_0$ holds for all $j\in \mathbb{N}$. In either case, there exists a subsequence $\{ s_{j_n}\}_{n=1}^\infty$ such that $\{\gamma(s_{j_n})\}_{n=1}^\infty$ converges to $\gl$ as $n \to \infty$. Moreover, the uniqueness of $s_0$ implies $s_{j_n}\to s_0$. Since $|\gamma(s)|$ achieves its supremum at $s_0$, Assumption \ref{assum:R} (iv) implies  $R(s_0)>0$, and hence, $R(s_{j_n})>0$ for all large $n$. Now we analyze $\partial_sQ(s,r)$ and $\gamma'(s)$ under the condition $R(s)>0$. Note that 
 \begin{align}
 	\partial_{s} Q\left( s,r \right) &= \left( r-1 \right) \left( R\left( s \right) \right)^{r-2} \left( R^{\prime}\left( s \right) \left( e^{-s}-R\left( s \right) \right) -\frac{R\left( s \right)}{r-1} \cdot \left( e^{-s}+R^{\prime}\left( s \right) \right) \right) \nonumber \\
 	&= \left( r-1 \right) \left( R\left( s \right) \right)^{r-2} \left( \left( 1-R\left( s \right) \right)^{2} \gamma^{\prime} \left( s \right) +\left( e^{-s}+R^{\prime}\left( s \right) \right) \left( 1-R\left( s \right) -\frac{R(s)}{r-1} \right) \right). \label{eqn:Q,gamma}
 \end{align}
 At $s=s_{j_n}$ and $r= i(j_n)/j_n$, we have $\partial_sQ(s,r)=0$. Under Assumption \ref{assum:R} (iv) and the condition $s_{j_n} \to s_0$ as $n \to \infty$,
 there exist $n_0>0$ and $A_0>0$ such that, for any $n>n_0$, we have $|e^{-s_{j_{n}}}+R'(s_{j_n})|\geq A_0$, $(1-R(s_{j_n}))^2\geq A_0$ and $R(s_{j_n}) \geq A_0$. Thus, the condition $\partial_s Q(s_{j_n},i(j_n)/j_n)=0$ implies
 \begin{equation*}
 	\gamma^{\prime} \left( s_{j_n} \right) =\frac{e^{-s_{j_n}}+R^{\prime}\left( s_{j_n} \right)}{\left( 1-R\left( s_{j_n} \right) \right)^{2}} \cdot \left( 1-R\left( s_{j_n} \right) -\frac{R\left( s_{j_n} \right)}{i(j_n)/j_n-1} \right).
 \end{equation*}
 If $\gamma'(s_{j_n}) \to 0$, then $h(i(j_n)/j_n) \to 0$ as $n \to \infty$, which contradicts the condition $h(i(j_n)/j_n)\geq h_0>0$ when $i(j_n) \in I_Kj_n\cap \mathbb{N}$, and hence the claim \eqref{eqn:c0-gamma} follows.
 Now by the claim \eqref{eqn:c0-gamma} and the estimate \eqref{eqn:basic-ineq0}, we have 
 {\begin{equation*}
 		\begin{aligned}
 			\left|
 			R(s_{i,j})^{i-j} \left( e^{-s_{i,j}} - R(s_{i,j}) \right)^j \right| &= 
 			\left|
 			R(s_{i,j})^{i-j} \left( 1- |R(s_{i,j})|\right)^j \right| \left|
 			\frac{e^{-s_{i,j}}- R(s_{i,j})}{1-|R(s_{i,j})|} \right|^j \\
 			& = c_0^j   \left(
 			|R(s_{i,j})|^{i-j} \left( 1- |R(s_{i,j})|\right)^j \right) 
 			\le c_0^j \frac{(i-j)^{i-j}j^j}{i^i}.
 		\end{aligned}
 \end{equation*}}
 Next, by Stirling's formula \eqref{eqn:stirling} and the assertion $c_0 < \gl$, we arrive at
 \begin{equation*}
 	\begin{aligned}
 		\sum_{\substack{i \in I_{K}j \cap \mathbb{N} \\ i \neq j}} 
 		\binom{i}{j}
 		\sup_{s\in(0,\infty)} \left|
 		R(s)^{i-j} \left( e^{-s} - R(s) \right)^j
 		\right| & \le \sum_{i=j+1}^{\lfloor r_0 j\rfloor+1} \binom{i}{j}
 		\left|
 		R(s_{i,j})^{i-j} \left( e^{-s_{i,j}} - R(s_{i,j}) \right)^j
 		\right| \\
 		& \le C c_0^j \sum_{i=j+1}^{\lfloor r_{0}j\rfloor +1} \sqrt{\frac{i}{j\left( i-j \right)}} \le C j c_0^j \le C\gl^j.
 	\end{aligned}
 \end{equation*}
 \textbf{Step (iii):}
 Let $B_k=(r_k^\ast - \delta,r_k^\ast + \delta)$, $k=1,\cdots,K$. Consider the case $i \in (\bigcup_{k=1}^K B_k) \cap \mathbb{N}$,
 \begin{align*}
 	{\rm III}
 	&= \sum_{k=1}^{K} \sum_{i \in B_{k}j \cap \mathbb{N}} \binom{i}{j} \sup_{s \in (0,\infty)} \left| \big(R(s)\big)^{i-j} \big(e^{-s} - R(s)\big)^j \right| \\
 	&\leq \gl^{j} \sum_{k=1}^{K} \sum_{i \in B_{k}j \cap \mathbb{N}} \binom{i}{j} \big|R(s_{i,j})\big|^{i-j} \big(1 - |R(s_{i,j})|\big)^j,
 \end{align*}
 where $s_{i,j}$ maximizes $|Q(s,i/j)|$. By Assumption \ref{assum:R} (iii) and Remark \ref{rem:ass}, there exists $a>0$ such that
 \begin{equation*}
 	|R(s_{i,j}) - (1-r_i^{-1})| \ge a |r_i-r_k^*|, \quad \forall {r_i=i/j}\in B_k.
 \end{equation*}
 Next, we bound for the case $R(s_{i,j})-(1-r_i^{-1})\geq a(r_k^\ast -r_i)$ when $r_k^\ast -\delta\leq r_i\leq r_k^\ast$ and other cases follow similarly. 
 Note that $r > 1$, and the map $ x\mapsto x^{r-1}(1-x)$ attains its maximum at $x = 1 - r^{-1}$ and decreases monotonically over the interval $[1 - r^{-1}, 1]$. Therefore,
 for all $r_i\in [r_k^\ast -\delta, r_k^\ast]$,
 \begin{equation}\label{eqn:R,1-R-rev}
 	\big|R(s_{i,j})\big|^{r-1} \big(1 - |R(s_{i,j})|\big)
 	\le  (1 - r_i^{-1} + a h_{i,k})^{r-1} (r_i^{-1}- a h_{i,k}),\quad \mbox{with }h_{i,k} = r_k^* - r_i. 
 \end{equation}
 This further implies
 \begin{equation*} 
 	\sum_{i \in B_{k}j \cap \mathbb{N}} \binom{i}{j} \big|R(s_{i,j})\big|^{i-j} \big(1 - |R(s_{i,j})|\big)^j \le C \sum_{i = \lfloor (r_k - \delta)j \rfloor + 1}^{\lfloor r_k j \rfloor} \binom{i}{j} \left[\left(1 - \frac{1}{r_i} + a h_{i,k}\right)^{r_i-1} \left(\frac{1}{r_i}- a h_{i,k}\right)\right]^j.
 \end{equation*}
 Let $g_k(r):=ar (r^\ast_k -r)$. By Stirling's formula \eqref{eqn:stirling}, we obtain the following bound on the summand:
 \begin{align*}
 	& \binom{i}{j} \left[\left(1 - \frac{1}{r_i} + a h_{i,k}\right)^{r_i-1} \left(\frac{1}{r_i}- a h_{i,k}\right)\right]^j
 	\leq \frac{C}{\sqrt{j}}  \frac{(r_i)^{i}}{(r_i - 1)^{i-j}} \left(\frac{r_i - 1 + a r_i h_{i,k}}{r_i}\right)^{i-j} \left(\frac{1 - a r_i h_{i,k}}{r_i}\right)^{j} \\
 	=& \frac{C }{\sqrt{j}}   \left(1 + \frac{g_{k}(r_i)}{r_i - 1}\right)^{i-j} \left(1 - g_{k}(r_i)\right)^{j} = \frac{C }{\sqrt{j}}  \left(\left(1 + \frac{g_{k}(r_i)}{r_i - 1}\right)^{r_i - 1} \left(1 - g_{k}(r_i)\right)\right)^{j}.
 \end{align*}
 For sufficiently small $\delta>0$, we have
 $$  \left(1 + \frac{g_{k}(r_i)}{r_i - 1}\right)^{r_i - 1} \big(1 - g_{k}(r_i)\big) \le e^{-C (h_{i,k})^2}, \quad \forall r_i \in (r_k^\ast-\delta,r_k^\ast).$$
 Now let $x_k^{(i)}=\frac{i-jr_k^\ast}{\sqrt{j}}$ and note the identity $(x_k^{(i)})^2= j^{-1}(i-jr_k^\ast)^2 = j (h_{i,k})^2 $. Then we derive
 \begin{align*}
 	&\sum_{i = \lfloor (r_{k} - \delta)j \rfloor + 1}^{\lfloor r_{k}j \rfloor} \left(\left(1 + \frac{g_{k}(r_i)}{r_i - 1}\right)^{r_i - 1} \left(1 - g_{k}(r_i)\right)\right)^{j}  \le \sum_{i = \lfloor (r_{k} - \delta)j \rfloor + 1}^{\lfloor r_{k}j \rfloor} e^{-C j (h_{i,k})^2} = \sum_{i = \lfloor (r_{k} - \delta)j \rfloor + 1}^{\lfloor r_{k}j \rfloor} e^{-C (x_k^{(i)})^2} \\
 	=& \sum_{i = \lfloor (r_{k} - \delta)j \rfloor + 1}^{\lfloor r_{k}j \rfloor} e^{-C (x_k^{(i)})^2} \sqrt{j} (x_{k}^{(i+1)} - x_{k}^{(i)}) 
 	\le C \sqrt{j }\int_0^\infty e^{-Cx^2} \, \d x \le C \sqrt{j }.
 \end{align*}
 Thus we obtain ${\rm III} \le C K \gl^j$.
 This completes the proof of the theorem.
 $\qed$

 \section{Two auxiliary results}
 \begin{lemma}[Discrete Gronwall's  inequality]\label{lem:x_k_1}
 	Consider the following recurrence inequality for $n\geq 1$ and $x_n>0$:
 	\begin{align*}
 		x_{n+1} \le C \Delta T \sum_{i=1}^{n} T_{n-i+1}^{-\frac12} x_{i} + a,\quad \text{with}~a>0.
 	\end{align*}
 	Then there exists a constant $C$  independent of $n$ and $a$ such that
 	$x_{n+1} \leq C\sqrt{T} a e^{CT}$.
 \end{lemma}
 \begin{proof}
 	By iterating the given inequality once, the inequality becomes
 	\begin{align*}
 		x_{n+1}  \le C \Delta T \sum_{i=1}^{n}\frac{1}{\sqrt{(n-i+1)\Delta T}} x_{i} + a \leq C\Delta T\sum_{j=1}^{n-1}b_jx_j + aC\sqrt{\Delta T}\sum_{m=1}^n \frac{1}{\sqrt m} + a,
 	\end{align*}
 	where $b_j = \sum_{k=1}^{n-j} \frac{1}{\sqrt{n-j-k+1}}\frac{1}{\sqrt{k}}$ and $b_j \leq 4$ holds for all $j$. Then the inequality becomes
 	\begin{align*}
 		x_{n+1} &\leq C\Delta T\sum_{j=1}^{n-1} x_j + a(C\sqrt{T_n}+1) \leq C\Delta T\sum_{j=1}^{n-1} x_j +a C\sqrt{T}.
 	\end{align*}
 	Then we can conclude the lemma using the standard Gronwall's inequality.
 \qed \end{proof}
 
 \begin{lemma}\label{lem:x_k}
 	Consider the following recurrence inequality for $x_k > 0$:
 	\begin{equation*}
 		x_{k+1} \leq \gamma^{k+1} f_0 + h \sum_{j=0}^k \gamma^j x_{k-j}, \quad \text{with } h, \gamma < 1.
 	\end{equation*}
 	Then for  $M =\max \{f_0, x_0 \} > 0 $, there holds
 	\begin{equation*}
 		x_k \leq M \cdot (\gamma + h)^k, \quad \forall  k \geq 0.
 	\end{equation*}
 \end{lemma}
 \begin{proof}
 	The proof is based on induction. For $k=0$, by the definition of $M$, $M\geq f_0$. Thus
 	\begin{equation*}
 		x_1 \leq \gamma f_0 + h x_0 \leq \gamma f_0 + h M \leq M(\gamma + h).
 	\end{equation*}
 	Now suppose that the estimate $ x_j \leq M(\gamma + h)^j $ holds for all $ j \leq k $ for some $k$. Then we obtain
 	\begin{equation*}
 		x_{k+1} \leq \gamma^{k+1} f_0 + h \sum_{j=0}^k \gamma^j \cdot M(\gamma + h)^{k-j}.
 	\end{equation*}
 	Meanwhile, by the geometric series sum 
 	\begin{equation*}
 		\sum_{j=0}^{k} \gamma^{j} \left( \gamma +h \right)^{k-j} =\left( \gamma +h \right)^{k} \sum_{j=0}^{k} \left( \frac{\gamma}{\begin{gathered}\gamma +h\end{gathered}} \right)^{j} =\left( \gamma +h \right)^{k+1} \frac{1-r^{k+1}}{h},
 	\end{equation*}
 	where $r=\frac{\gamma}{\gamma + h}$, we have
 	\begin{align*}
 		x_{k+1} &\leq \gamma^{k+1} f_0 + M (\gamma +h)^{k+1} (1-r^{k+1}) \\
 		&= M (\gamma+h)^{k+1} r^{k+1} + M (\gamma +h)^{k+1} (1-r^{k+1})=M(\gamma +h)^{k+1}.
 	\end{align*}
 	This completes the proof of the lemma.
 \qed \end{proof}

\bibliographystyle{spmpsci}
\bibliography{reference}

@article{jin2025optimizing,
	title={Optimizing Coarse Propagators in Parareal Algorithms},
	author={Jin, Bangti and Lin, Qingle and Zhou, Zhi},
	journal={SIAM J. Sci. Comput.},
	volume={47},
	number={2},
	pages={A735--A761},
	year={2025},
	publisher={SIAM}
}

@article {GanderKwok:2020,
    AUTHOR = {Gander, Martin J. and Kwok, Felix and Salomon, Julien},
     TITLE = {Para{O}pt: a parareal algorithm for optimality systems},
   JOURNAL = {SIAM J. Sci. Comput.},
  FJOURNAL = {SIAM Journal on Scientific Computing},
    VOLUME = {42},
      YEAR = {2020},
    NUMBER = {5},
     PAGES = {A2773--A2802},
      ISSN = {1064-8275,1095-7197},
   MRCLASS = {49M27 (49M05 49M41 65F08 65K10 65Y05 68W10)},
  MRNUMBER = {4150714},
       DOI = {10.1137/19M1292291},
       URL = {https://doi-org.ezproxy.lb.polyu.edu.hk/10.1137/19M1292291},
}

@article {Fang:2022,
    AUTHOR = {Fang, Liang and Vandewalle, Stefan and Meyers, Johan},
     TITLE = {A parallel-in-time multiple shooting algorithm for large-scale
              {PDE}-constrained optimal control problems},
   JOURNAL = {J. Comput. Phys.},
  FJOURNAL = {Journal of Computational Physics},
    VOLUME = {452},
      YEAR = {2022},
     PAGES = {Paper No. 110926, 24},
      ISSN = {0021-9991,1090-2716},
   MRCLASS = {65M06 (49M41 65L10 65Y05)},
  MRNUMBER = {4361834},
       DOI = {10.1016/j.jcp.2021.110926},
       URL = {https://doi-org.ezproxy.lb.polyu.edu.hk/10.1016/j.jcp.2021.110926},
}

@incollection {BalMaday:2002,
	AUTHOR = {Bal, Guillaume and Maday, Yvon},
	TITLE = {A ``parareal'' time discretization for non-linear {PDE}'s with
	application to the pricing of an {A}merican put},
	BOOKTITLE = {{Recent Developments in Domain Decomposition Methods
	({Z}\"{u}rich, 2001)}},
	SERIES = {Lect. Notes Comput. Sci. Eng.},
	VOLUME = {23},
	PAGES = {189--202},
	PUBLISHER = {Springer, Berlin},
	YEAR = {2002},
	ISBN = {3-540-43413-5},
	MRCLASS = {65M06 (91B28)},
	MRNUMBER = {1962689},
	DOI = {10.1007/978-3-642-56118-4\_12},
	URL = {https://doi.org/10.1007/978-3-642-56118-4_12},
}

@article {PagesPironneau:2016,
	AUTHOR = {Pag\`es, Gilles and Pironneau, Olivier and Sall, Guillaume},
	TITLE = {The parareal algorithm for {A}merican options},
	JOURNAL = {C. R. Math. Acad. Sci. Paris},
	FJOURNAL = {Comptes Rendus Math\'{e}matique. Acad\'{e}mie des Sciences.
	Paris},
	VOLUME = {354},
	YEAR = {2016},
	NUMBER = {11},
	PAGES = {1132--1138},
	ISSN = {1631-073X,1778-3569},
	MRCLASS = {65C05 (65M75 65Y05 91G20 91G60)},
	MRNUMBER = {3566516},
	DOI = {10.1016/j.crma.2016.09.010},
	URL = {https://doi.org/10.1016/j.crma.2016.09.010},
}

@article {BrehierWang:2020,
	AUTHOR = {Brehier, Charles-Edouard and Wang, Xu},
	TITLE = {On parareal algorithms for semilinear parabolic stochastic
	{PDE}s},
	JOURNAL = {SIAM J. Numer. Anal.},
	FJOURNAL = {SIAM Journal on Numerical Analysis},
	VOLUME = {58},
	YEAR = {2020},
	NUMBER = {1},
	PAGES = {254--278},
	ISSN = {0036-1429,1095-7170},
	MRCLASS = {65M12 (60H15 60H35)},
	MRNUMBER = {4049401},
	MRREVIEWER = {Marco\ P.\ Cabral},
	DOI = {10.1137/19M1251011},
	URL = {https://doi.org/10.1137/19M1251011},
}

@article {ArielKim:2016,
	AUTHOR = {Ariel, Gil and Kim, Seong Jun and Tsai, Richard},
	TITLE = {Parareal multiscale methods for highly oscillatory dynamical
	systems},
	JOURNAL = {SIAM J. Sci. Comput.},
	FJOURNAL = {SIAM Journal on Scientific Computing},
	VOLUME = {38},
	YEAR = {2016},
	NUMBER = {6},
	PAGES = {A3540--A3564},
	ISSN = {1064-8275,1095-7197},
	MRCLASS = {65L05 (37N30 65Y05)},
	MRNUMBER = {3569557},
	MRREVIEWER = {Gabriela\ Schranz-Kirlinger},
	DOI = {10.1137/15M1011044},
	URL = {https://doi.org/10.1137/15M1011044},
}

@article {XuHesthaven:2015,
	AUTHOR = {Xu, Qinwu and Hesthaven, Jan S. and Chen, Feng},
	TITLE = {A parareal method for time-fractional differential equations},
	JOURNAL = {J. Comput. Phys.},
	FJOURNAL = {Journal of Computational Physics},
	VOLUME = {293},
	YEAR = {2015},
	PAGES = {173--183},
	ISSN = {0021-9991,1090-2716},
	MRCLASS = {65L05 (34A08 65L20 65L60 65Y05)},
	MRNUMBER = {3342465},
	MRREVIEWER = {Duygu\ D\"{o}nmez Demir},
	DOI = {10.1016/j.jcp.2014.11.034},
	URL = {https://doi.org/10.1016/j.jcp.2014.11.034},
}

@article {Engblom:2009,
	AUTHOR = {Engblom, Stefan},
	TITLE = {Parallel in time simulation of multiscale stochastic chemical
	kinetics},
	JOURNAL = {Multiscale Model. Simul.},
	FJOURNAL = {Multiscale Modeling \& Simulation. A SIAM Interdisciplinary
	Journal},
	VOLUME = {8},
	YEAR = {2009},
	NUMBER = {1},
	PAGES = {46--68},
	ISSN = {1540-3459,1540-3467},
	MRCLASS = {65C30 (60H35 60J22 60J75 65Y05)},
	MRNUMBER = {2575044},
	MRREVIEWER = {Gon\c{c}alo\ Dos Reis},
	DOI = {10.1137/080733723},
	URL = {https://doi.org/10.1137/080733723},
}

@article {GWZ:Acta,
	AUTHOR = {Gander, Martin J. and Wu, Shu-Lin and Zhou, Tao},
	TITLE = {Time parallelization for hyperbolic and parabolic problems},
	JOURNAL = {Acta Numer.},
	FJOURNAL = {Acta Numerica},
	VOLUME = {34},
	YEAR = {2025},
	PAGES = {385--489},
	ISSN = {0962-4929},
	MRCLASS = {65M55 (65M06 65M12 65M15 65Y05)},
	MRNUMBER = {4926314},
	DOI = {10.1017/S0962492924000072},
	URL = {https://doi.org/10.1017/S0962492924000072},
}

@article {FriedhoffSouthworth:2021,
	AUTHOR = {Friedhoff, Stephanie and Southworth, Ben S.},
	TITLE = {On ``optimal'' {$h$}-independent convergence of parareal and
	multigrid-reduction-in-time using {R}unge-{K}utta time
	integration},
	JOURNAL = {Numer. Linear Algebra Appl.},
	FJOURNAL = {Numerical Linear Algebra with Applications},
	VOLUME = {28},
	YEAR = {2021},
	NUMBER = {3},
	PAGES = {e2301, 30},
	ISSN = {1070-5325},
	MRCLASS = {65L06 (65L20 65Y05)},
	MRNUMBER = {4242247},
	MRREVIEWER = {Bernhard A. Schmitt},
	DOI = {10.1002/nla.2301},
	URL = {https://doi.org/10.1002/nla.2301},
}

@article {DeSterck:2021,
	AUTHOR = {De Sterck, Hans and Falgout, Robert D. and Friedhoff,
	Stephanie and Krzysik, Oliver A. and MacLachlan, Scott P.},
	TITLE = {Optimizing multigrid reduction-in-time and parareal
	coarse-grid operators for linear advection},
	JOURNAL = {Numer. Linear Algebra Appl.},
	FJOURNAL = {Numerical Linear Algebra with Applications},
	VOLUME = {28},
	YEAR = {2021},
	NUMBER = {4},
	PAGES = {Paper No. e2367, 22},
	ISSN = {1070-5325},
	MRCLASS = {65M20 (65L06 65M22 65Y05)},
	MRNUMBER = {4284091},
	MRREVIEWER = {Daniel R. Reynolds},
	DOI = {10.1002/nla.2367},
	URL = {https://doi.org/10.1002/nla.2367},
}

@article {Bossuyt:2025,
    AUTHOR = {Bossuyt, Ignace and Vandewalle, Stefan and Samaey, Giovanni},
     TITLE = {Monte {C}arlo-moments micro-macro {P}arareal method for
              unimodal and bimodal scalar {M}c{K}ean-{V}lasov {SDE}s},
   JOURNAL = {SIAM J. Sci. Comput.},
  FJOURNAL = {SIAM Journal on Scientific Computing},
    VOLUME = {47},
      YEAR = {2025},
    NUMBER = {6},
     PAGES = {A3239--A3275},
      ISSN = {1064-8275,1095-7197},
   MRCLASS = {65C30 (60H35 65C35)},
  MRNUMBER = {4980425},
       DOI = {10.1137/23M1609142},
       URL = {https://doi-org.ezproxy.lb.polyu.edu.hk/10.1137/23M1609142},
}

@article {GanderVandewalle:2007,
	AUTHOR = {Gander, Martin J. and Vandewalle, Stefan},
	TITLE = {Analysis of the parareal time-parallel time-integration
	method},
	JOURNAL = {SIAM J. Sci. Comput.},
	FJOURNAL = {SIAM Journal on Scientific Computing},
	VOLUME = {29},
	YEAR = {2007},
	NUMBER = {2},
	PAGES = {556--578},
	ISSN = {1064-8275,1095-7197},
	MRCLASS = {65R20 (65L20)},
	MRNUMBER = {2306258},
	MRREVIEWER = {Neide\ Bertoldi\ Franco},
	DOI = {10.1137/05064607X},
	URL = {https://doi.org/10.1137/05064607X},
}

@book{thomee2007galerkin,
	AUTHOR = {Thom\'{e}e, Vidar},
	TITLE = {{Galerkin Finite Element Methods for Parabolic Problems}},
	EDITION = {Second},
	PUBLISHER = {Springer-Verlag, Berlin},
	YEAR = {2006},
	PAGES = {xii+370},
	ISBN = {978-3-540-33121-6; 3-540-33121-2},
	MRCLASS = {65-02 (65M15 65M60)},
	MRNUMBER = {2249024},
}

@book{martin-pint,
	author = {Gander, Martin J. and Lunet, Thibaut},
	title = {{Time Parallel Time Integration}},
	publisher = {SIAM, Philadelphia, PA},
	year = {2024},
	doi = {10.1137/1.9781611978025},
	edition   = {},
	doi = {10.1137/1.9781611978025},
}

@article {LionsMadayTurinici:2001,
	AUTHOR = {Lions, Jacques-Louis and Maday, Yvon and Turinici, Gabriel},
	TITLE = {R\'{e}solution d'{EDP} par un sch\'{e}ma en temps
	``parar\'{e}el''},
	JOURNAL = {C. R. Acad. Sci. Paris S\'{e}r. I Math.},
	FJOURNAL = {Comptes Rendus de l'Acad\'{e}mie des Sciences. S\'{e}rie I.
	Math\'{e}matique},
	VOLUME = {332},
	YEAR = {2001},
	NUMBER = {7},
	PAGES = {661--668},
	ISSN = {0764-4442},
	MRCLASS = {65M06 (65M20 65Y05)},
	MRNUMBER = {1842465},
	MRREVIEWER = {Dennis\ C.\ Jespersen},
	DOI = {10.1016/S0764-4442(00)01793-6},
	URL = {https://doi.org/10.1016/S0764-4442(00)01793-6},
}

@article{Dobrev:2017,
	AUTHOR = {Dobrev, V. A. and Kolev, Tz. and Petersson, N. A. and
	Schroder, J. B.},
	TITLE = {Two-level convergence theory for multigrid reduction in time
	({MGRIT})},
	JOURNAL = {SIAM J. Sci. Comput.},
	FJOURNAL = {SIAM Journal on Scientific Computing},
	VOLUME = {39},
	YEAR = {2017},
	NUMBER = {5},
	PAGES = {S501--S527},
	ISSN = {1064-8275},
	MRCLASS = {65M55 (65F10 65M22 65Y05)},
	MRNUMBER = {3716570},
	MRREVIEWER = {Benjamin Wi-Lian Ong},
	DOI = {10.1137/16M1074096},
	URL = {https://doi-org.ezproxy.lb.polyu.edu.hk/10.1137/16M1074096},
}

@incollection{gander201550,
	title={50 years of time parallel time integration},
	author={Gander, Martin J},
	booktitle={{Multiple Shooting and Time Domain Decomposition Methods: MuS-TDD, Heidelberg, May 6-8, 2013}},
	pages={69--113},
	year={2015},
	publisher={Springer}
}

@incollection{gander2008nonlinear,
	title={Nonlinear convergence analysis for the parareal algorithm},
	author={Gander, Martin J and Hairer, Ernst},
	booktitle={{Domain Decomposition Methods in Science and Engineering XVII}},
	pages={45--56},
	year={2008},
	publisher={Springer}
}

@article{WuZhou:2015,
	AUTHOR = {Wu, Shu-Lin and Zhou, Tao},
	TITLE = {Convergence analysis for three parareal solvers},
	JOURNAL = {SIAM J. Sci. Comput.},
	FJOURNAL = {SIAM Journal on Scientific Computing},
	VOLUME = {37},
	YEAR = {2015},
	NUMBER = {2},
	PAGES = {A970--A992},
	ISSN = {1064-8275},
	MRCLASS = {65L05 (65L20 65M20 65Y05 68Q60)},
	MRNUMBER = {3338002},
	MRREVIEWER = {Philip W. Sharp},
	DOI = {10.1137/140970756},
	URL = {https://doi-org.ezproxy.lb.polyu.edu.hk/10.1137/140970756},
}

@article{wu2015convergence2,
	title={Convergence analysis of some second-order parareal algorithms},
	author={Wu, Shu-Lin},
	journal={IMA J. Numer. Anal.},
	volume={35},
	number={3},
	pages={1315--1341},
	year={2015},
	publisher={OUP}
}

@incollection{bal2005convergence,
	title={On the convergence and the stability of the parareal algorithm to solve partial differential equations},
	author={Bal, Guillaume},
	booktitle={{Domain Decomposition Methods in Science and Engineering}},
	pages={425--432},
	year={2005},
	publisher={Springer}
}

@article{hessenthaler2020multilevel,
	title={Multilevel convergence analysis of multigrid-reduction-in-time},
	author={Hessenthaler, Andreas and Southworth, Ben S and Nordsletten, David and R\"ohrle, Oliver and Falgout, Robert D and Schroder, Jacob B},
	journal={SIAM J. Sci. Comput.},
	volume={42},
	number={2},
	pages={A771--A796},
	year={2020},
	publisher={SIAM}
}

@article{gander2013analysis,
	title={Analysis of two parareal algorithms for time-periodic problems},
	author={Gander, Martin J and Jiang, Yao-Lin and Song, Bo and Zhang, Hui},
	journal={SIAM J. Sci. Comput.},
	volume={35},
	number={5},
	pages={A2393--A2415},
	year={2013},
	publisher={SIAM}
}

@article{legoll2013micro,
	title={A micro-macro parareal algorithm: application to singularly perturbed ordinary differential equations},
	author={Legoll, Fr{\'e}d{\'e}ric and Lelievre, Tony and Samaey, Giovanni},
	journal={SIAM J. Sci. Comput.},
	volume={35},
	number={4},
	pages={A1951--A1986},
	year={2013},
	publisher={SIAM}
}

@book{suhubi2013functional,
	title={Functional analysis},
	author={Suhubi, Erdogan},
	year={2013},
	publisher={Springer, Dordrecht}
}

@article {Li:2024,
	AUTHOR = {Li, Guanglian},
	TITLE = {Wavelet-based edge multiscale parareal algorithm for
	subdiffusion equations with heterogeneous coefficients in a
	large time domain},
	JOURNAL = {J. Comput. Appl. Math.},
	FJOURNAL = {Journal of Computational and Applied Mathematics},
	VOLUME = {440},
	YEAR = {2024},
	PAGES = {115608, 24 pp.},
	ISSN = {0377-0427,1879-1778},
	MRCLASS = {65M70 (65T60)},
	MRNUMBER = {4660615},
	DOI = {10.1016/j.cam.2023.115608},
	URL = {https://doi.org/10.1016/j.cam.2023.115608},
}

@book{spencer2014asymptopia,
	title={Asymptopia},
	author={Spencer, Joel},
	year={2014},
	publisher={AMS, Providence, RI}
}

@article{nirenberg1966extended,
	AUTHOR = {Nirenberg, L.},
	TITLE = {An extended interpolation inequality},
	JOURNAL = {Ann. Scuola Norm. Sup. Pisa Cl. Sci. (3)},
	FJOURNAL = {Annali della Scuola Normale Superiore di Pisa. Classe di
	Scienze. Serie III},
	VOLUME = {20},
	YEAR = {1966},
	PAGES = {733--737},
	ISSN = {0391-173X},
	MRCLASS = {46.38},
	MRNUMBER = {208360},
	MRREVIEWER = {Richard\ Beals},
}

@book{adams2003sobolev,
	title={{Sobolev Spaces}},
	author={Adams, Robert A and Fournier, John JF},
	year={2003},
	publisher={Elsevier, Amsterdam}
}

@article{dai2013stable,
	title={Stable parareal in time method for first-and second-order hyperbolic systems},
	author={Dai, Xiaoying and Maday, Yvon},
	journal={SIAM Journal on Scientific Computing},
	volume={35},
	number={1},
	pages={A52--A78},
	year={2013},
	publisher={SIAM}
}

@article{pentland2023error,
	title={Error bound analysis of the stochastic parareal algorithm},
	author={Pentland, Kamran and Tamborrino, Massimiliano and Sullivan, Timothy John},
	journal={SIAM Journal on Scientific Computing},
	volume={45},
	number={5},
	pages={A2657--A2678},
	year={2023},
	publisher={SIAM}
}

@article{gander2019new,
	title={A new parareal algorithm for problems with discontinuous sources},
	author={Gander, Martin J and Kulchytska-Ruchka, Iryna and Niyonzima, Innocent and Scho\"{o}ps, Sebastian},
	journal={SIAM J. Sci. Comput.},
	volume={41},
	number={2},
	pages={B375--B395},
	year={2019},
	publisher={SIAM}
}

\end{document}